\documentclass[final]{svjour3_edit}
\usepackage{amsmath, amssymb, latexsym}

\usepackage[english]{babel}
\usepackage[latin1]{inputenc}
\usepackage{tikz}
\usepackage{subcaption}
\usepackage{lscape}
\usepackage[justification=centering]{caption}
\usepackage{listings}
\numberwithin{equation}{section}

\usepackage{hyperref}

\newcommand{\enorm}[1]{\|#1\|_B}

\def\Oh{\mathcal{O}}
\def\Eps{\varepsilon}
\usepackage{color}
\usepackage{listings}
\newcommand{\listbox}[1]{\lstset{language=matlab,
    showstringspaces=false, 
    basicstyle=\footnotesize,
    frame=single,% backgroundcolor=\color{BG},
    texcl=true,
    linewidth=#1\textwidth, %numbers=left,
    numbersep=5pt,
    numberblanklines=false, 
    xleftmargin=1cm,
    keywordstyle=\color{black}\bfseries}
}
\def\khat{\widetilde{k}}

\title{An introduction to the analysis and implementation of sparse grid
  finite element methods} 
  \titlerunning{Introduction to analysis and implementation of sparse grid FEMs}
\author{Stephen Russell \and Niall Madden}
\institute{S. Russell \at
School of Mathematics, Statistics and Applied Mathematics, \\
National University of Ireland, Galway, Ireland\\
\email{S.Russell1@NUIGalway.ie}
\and
N. Madden  \at
School of Mathematics, Statistics and Applied Mathematics, \\
National University of Ireland, Galway, Ireland\\
\email{Niall.Madden@NUIGalway.ie}}
\date{October, 2015}
\begin{document}

\maketitle

\begin{abstract}
Our goal  is to present an elementary approach to the
analysis and programming of sparse grid finite element methods. This
family of schemes can compute accurate solutions to partial
differential equations, but using far fewer degrees of freedom than
their classical counterparts.
After a brief discussion of the
classical Galerkin finite element method with bilinear elements, we
give a short analysis of what is probably the simplest sparse grid
method: the  two-scale technique  of Lin et
al. \cite{LiYa01}. We  
then demonstrate  how to extend this to a \emph{multiscale} sparse
grid method which, up to choice of basis, is equivalent to the hierarchical
approach,  
as described in, e.g., \cite{BuGr04}. However, by presenting it as 
an extension of the
two-scale method, we can give an elementary treatment of its analysis
and implementation. 
For each method considered, we provide
MATLAB code, and a comparison of accuracy and  
computational costs.


\end{abstract}

\section{Introduction}

Sparse grid methods provide a means of approximating functions and
data in a way that avoids the notorious ``curse of dimensionality'':
for fixed accuracy, the computational effort required by classical
methods grows exponentially in the number of dimensions. Sparse grid
methods hold out the hope of retaining the accuracy of classical
techniques, but at a cost that is essentially independent of the
number of dimensions.

An important application of sparse grid methods is the solution of
partial differential equations (PDEs) by finite element methods
(FEMs). Naturally, the solution to a PDE is found in an infinite
dimensional space. A FEM first reformulates the problem as an integral
equation, and then restricts this problem to a suitable 
finite-dimensional subspace; most typically,  this subspace is  comprised of
piecewise polynomials.  This  ``restricted'' problem can be
expressed as a matrix-vector equation; solving it gives 
the finite element approximation to the solution of the PDE. For
many FEMs, including  those considered in this article, this
solution is the best possible approximation (with respect to a certain
norm)  that
one can find in the finite-dimensional subspace. That is, the accuracy
of the FEM solution depends on the approximation properties of the
finite dimensional space.

The main computational cost incurred by a FEM is in the solution of
the matrix-vector equation. This linear system is sparse, and
amenable to solution by direct methods, such as Cholesky-factorisation,
if it is not too large, or  by highly efficient iterative methods,
such as multigrid methods, for larger systems. The order of the system
matrix is the dimension of the finite-element space, and so great
computational efficiencies can be gained over classical FEMs by
constructing a finite-dimensional space of reduced size without
compromising the approximation properties: this is what is achieved
by sparse grid methods.

We will consider the numerical solution of the following PDE
\begin{subequations}
\label{eq:model}
 \begin{equation}
 \label{eq:2D problem}
  Lu:=-\Delta u + ru=f(x,y) \quad \text{in }\Omega :=(0,1)^2,
 \end{equation}
\begin{equation}
\label{eq:2D BCs}
 u=0 \quad \text{on } \partial \Omega.
\end{equation}
\end{subequations}
Here $r$ is a positive constant and so, for simplicity, we shall take
$r=1$, but $f$ is an arbitrary function. Our
choice of problem is motivated by the desire, for the purposes of
exposition, to keep the setting as simple as possible, but without
trivialising it. The main advantages of choosing $r$ to be constant
are that no quadrature is required when computing the system matrix,
and that accurate solutions can be obtained using a uniform mesh.
Thus, the construction of the system matrix for the standard Galerkin
FEM is reduced to
several lines of code, which  we
show how to do in Section
\ref{sec:Galerkin}. In Section 
\ref{sec:two-scale} we show how to develop the two-scale sparse grid
method and, in Section \ref{sec:multiscale} show how to extend this to
a multiscale setting. A comparison of the accuracy and efficiency of
the methods is given in Section \ref{sec:good times}. We stress that
the choice of constant $r$ in \eqref{eq:model} is only to simplify
the implementation of the standard  Galerkin method: our implementation
of the sparse grid methods is identical for variable $r$, and
would require only minor modifications for nonuniform tensor
product meshes.

We aim to provide an introduction to the mathematical analysis
and computer implementation of 
sparse grids that is accessible to readers who have a basic knowledge of 
finite element methods. 
%It should be accessible to an advanced undergraduate
%or beginning graduate student with a background in computational
%mathematics. 
For someone who is new to FEMs, we propose \cite[Chapter
14]{SuMa03} as a primer for the key concepts. An extensive mathematical treatment is given
in \cite{BrSc08}: we use results from its early chapters. 

We do not aim to present a comprehensive overview of
the state of the art sparse grid methods: there are many important
contributions to this area which we do not cite. However, we hope
that, having read this article, and experimented with the MATLAB
programs provided, the reader will be motivated to learn more from
important references in the area, such as \cite{BuGr04}.

\subsection{MATLAB code}

All our numerical examples are implemented in MATLAB. Full source code is available
from
\href{http://www.maths.nuigalway.ie/~niall/SparseGrids}
{\texttt{www.maths.nuigalway.ie/$\sim$niall/SparseGrids}} while snippets are
presented in the text.
Details
of  individual functions are given  in the relevant sections. 
For those  using MATLAB for the first time, we recommend \cite{Dris09}
as a readable introduction. Many of
the finer  points of programming in MATLAB are presented in \cite{Mole04} and
\cite{HiHi05}. A detailed study of programming FEMs in MATLAB may be found
in \cite{Gock06}.

We make use of the freely-available Chebfun toolbox
\cite{DrHa14} to allow the user to choose their own test problem, for
which the corresponding $f$ is automatically computed. We also use 
Chebfun for some  computations related to calculating the error. (We highly
recommend using Chebfun Version 5 or later, as some of the operations we use are  significantly faster
than in earlier versions). Only \href{http://www.maths.nuigalway.ie/~niall/SparseGrids/Test_FEM.m}{\texttt{Test\_FEM.m}}, 
the main test harness, 
uses Chebfun. That script contains comments to  
show how it may be modified to avoid using
Chebfun, and thus run on Octave \cite{octave:2014}. We have found the
direct linear  solver (``backslash'') to be less efficient in Octave
than MATLAB, and thus timing may be qualitatively different from those
presented here (though no doubt some optimisations are possible).

\subsection{Notation}
We use the standard $L_2$ inner product and norm:
\begin{equation*}
 (u, v) := \int_\Omega uv \,d\Omega, \qquad 
\|u\|_{0,\Omega} = \sqrt{(u,u)}.
\end{equation*}
The bilinear form, $B(\cdot,\cdot)$ associated with~\eqref{eq:model} is 
\begin{equation*}
 B(u,v)= (\nabla u,\nabla v) + (ru,v),
\end{equation*}
which induces the energy norm
\begin{equation}
 \label{eq:energynorm}
 \enorm{u} = \{\|\nabla u\|_{0,\Omega}^2+ \|u\|_{0,\Omega}^2\}^{1/2}.
\end{equation}

The space of functions whose (weak)
$k^\mathrm{th}$  derivatives are integrable on a domain $\omega$
is denoted by $H^k(\omega)$, and if these functions also 
 vanish on the
boundary of $\omega$ it is $H^k_0(\omega)$. See, e.g.,
\cite[Chap. 2]{BrSc08}  for more formal definitions.

Throughout this paper the letter $C$, with or without subscript,
denotes a generic  
positive constant that is independent of the discretization parameter,
$N$, and the scale of the method 
$k$, and may stand for different values in different places, even
within a single equation or inequality.

\section[standard Galerkin FEM]{Standard Galerkin finite element method}
\label{sec:Galerkin}

\subsection{A Galerkin FEM with bilinear elements}
The weak form of~(\ref{eq:model}) with $r=1$ is: \emph{find u $\in
  H^1_0(\Omega)$} 
such that
 \begin{equation}
 \label{eq:Bilinear} 
  B(u,v):= (\nabla u, \nabla v) + (u,v) = (f,v) \quad
  \forall v \in H^1_0(\Omega). 
 \end{equation}
A Galerkin FEM is obtained by replacing  $H_0^1(\Omega)$ in
\eqref{eq:Bilinear} with a suitable finite-dimensional subspace. 
We will take this to be
the space of piecewise bilinear functions  on a uniform mesh with $N_x$
equally sized intervals in one coordinate direction and $N_y$ equally sized intervals in the other.
We first form a one-dimensional mesh 
$\omega_x=\{x_0, x_1, \dots, x^{}_{N_x}\}$, where $x_i=i/N_x$. 
Any piecewise linear function on this mesh
can be uniquely expressed in terms of the so-called ``hat'' functions
 \begin{equation}
 \label{eq:1Dbasis}
  \psi^N_{i}(x) =
   \begin{cases}
    \displaystyle \frac{x-x_{i-1}^{}}{x_i-x_{i-1}} & \text{if }
    x_{i-1}^{} \leq  x<x_i, \\
    \displaystyle \frac{x_{i+1}^{}-x}{x_{i+1}-x_i}& \text{if }
    x_i^{}\leq x<x_{i+1}^{}, \\
        0       & \text{otherwise.}
   \end{cases}
 \end{equation}
Similarly, we can construct a mesh in the $y$-direction:
$\omega_y=\{y_0, y_1, \dots, y^{}_{N_y}\}$, where $y_j=j/N_y$. We can form a
two-dimensional $N_x\times N_y$ mesh by taking the Cartesian product of
$\omega_x$ and $\omega_y$, and then define  
 the space of piecewise bilinear functions,  $V^{}_{N_x,N_y}
 \subset H^1_0(\Omega)$ 
 as
 \begin{equation}
\label{eq:full basis}
  V^{}_{N_x,N_y}  = \text{span}
  \left\{\psi^{N_x}_{i}(x)\psi^{N_y}_{j}(y)\right\}^{i=1:N_x-1}_{j=1:N_y-1},
\end{equation}
where here we have adopted the compact MATLAB notation $i=1\!:\!N_x-1$ meaning
$i=1, 2, \dots, N_x-1$. (Later we  use expressions such as 
$i=1\!:\!2\!:\!N_x-1$ meaning $i=1, 3, 5, \dots, N_x-1$).

For sparse grid methods, we will be interested in meshes
where $N_x \neq N_y$. However, for the classical  Galerkin method,  we
%begin with in Section \ref{sec:Galerkin}, we will
take $N_x=N_y=N$. Thus, our finite element method for  \eqref{eq:model} is:
 \emph{find $u^{}_{N,N} \in {V}^{}_{N,N}$ such that}
\begin{equation}
\label{eq:Galerkin}
 B(u_{N,N}^{},v^{}_{N,N}) = (f,v^{}_{N,N})
 \quad \text{ for all } v^{}_{N,N} \in {V}^{}_{N,N}.
\end{equation}

\subsection{Analysis of the Galerkin FEM}
\label{sec:analysis}

%\emph{Sections \ref{sec:analysis} and \ref{sec:implementation} need to
 % be rearranged, or perhaps combined}.

The bilinear form defined in \eqref{eq:Bilinear} is continuous and
coercive, so \eqref{eq:Galerkin} possesses a unique
solution. Moreover, as noted in \cite[Section  3]{LMSZ09}, one has the
quasi-optimal bound:
\[
\enorm{u - u^{}_{N,N}}
\leq C \inf_{\psi \in V^{}_{N,N}(\Omega)} \enorm{u-\psi}.
%\leq C \enorm{u-I_{N,N}^{} u}.
\]
To complete the analysis, we can choose a good approximation of $u$ in 
$V^{}_{N,N}$. A natural choice is the nodal
interpolant of $u$. To be precise, let
$I_{N,N}:C(\bar{\Omega})\rightarrow V^{}_{N,N}$ be the
nodal piecewise bilinear interpolation operator that projects onto
$V^{}_{N,N}$. 
Since $I^{}_{N,N}u \in V^{}_{N,N}$,  one can complete the 
error analysis using the following classical results. For a derivation, see, e.g.,
\cite{GGRS07}. 
\begin{lemma}
\label{interpbounds}
 Let  $\tau$ be any mesh rectangle of size $h \times h$. Let $u\in
 H^2(\tau)$. Then its piecewise  bilinear nodal interpolant,
 $I^{}_{N,N}u$,
 satisfies the bounds
 \begin{subequations}
\begin{equation}
\label{eq:interpbounds1}
 \|u-I^{}_{N,N}u\|_{0,\Omega} \leq Ch^2( \|u_{xx}\|_{0,\Omega} + \|u_{xy}\|_{0,\Omega}
 + \|u_{yy}\|_{0,\Omega}),
 \end{equation}
 \begin{equation}
 \label{eq:interpbounds2}
 \|(u-I^{}_{N,N}u)_x\|_{0,\Omega} \leq C h(\|u_{xx}\|_{0,\Omega} + \|u_{xy}\|_{0,\Omega}),
 \end{equation}
 \begin{equation}
 \label{eq:interpbounds3}
\|(u-I^{}_{N,N}u)_y\|_{0,\Omega} \leq Ch(\|u_{xy}\|_{0,\Omega} + \|u_{yy}\|_{0,\Omega}).
\end{equation}
\end{subequations}
\end{lemma}
%A detailed proof of~(\ref{eq:interpbounds1}) is given by Brenner and Scott in \cite[Ch.\ 4]{BrSc08}
%while a proof of~(\ref{eq:interpbounds2}) and~(\ref{eq:interpbounds3}) can be found in \cite{Dura06}.
The following results follow directly from Lemma~\ref{interpbounds}.
\begin{lemma}
 \label{lem:uINN}
 Suppose $\Omega$= $(0,1)^2$. Let $u\in H^2_0(\Omega)$ and $I^{}_{N,N}u$
 be its piecewise bilinear nodal interpolant.
 Then there exists a constant $C$, independent of $N$, such that
\begin{equation*}
%\label{eq:interp bounds}
\|u-I^{}_{N,N}u \|_{0,\Omega} \leq CN^{-2}, 
\quad \text{ and } \quad 
\| \nabla(u-I^{}_{N,N}u)\|_{0,\Omega} \leq CN^{-1}.
\end{equation*}
Consequently, $\enorm{u-I_{N,N}u} \leq CN^{-1}$.
\end{lemma}
It follows immediately from these results 
that there exists a constant $C$, independent of $N$, such that
\begin{equation}
\label{eq:enormGalerkinFEM}
  \enorm{u-u^{}_{N,N}} \leq CN^{-1}.
\end{equation}

\subsection{Implementation of the Galerkin FEM}
\label{sec:Galerkin implementation}

To implement the method \eqref{eq:Galerkin}, we need to construct and
solve a linear system of equations. A useful FEM program also requires
ancillary tools, for example, to visualise the solution and to
estimate  errors.

Because the test problem we have chosen has a constant left-hand side,
the system matrix can be constructed in a few lines, and without
resorting to numerical quadrature. Also, because elements of the space
defined in \eqref{eq:full basis} are expressed as products of
one-dimensional functions,  the system
matrix can be expressed in terms of (Kronecker) products of matrices
arising from discretizing one-dimensional problems. We now explain how
this can be done. We begin with the one-dimensional analogue of
\eqref{eq:model}:
\[
-u''(x) + u(x)=f(x) \text{ on } (0,1),
\quad \text{ with } \quad 
u(0)=u(1)=1.
\]
Its finite element formulation is:  \emph{find $u^{}_{N} \in {V}^{}_{N}(0,1)$ such that}
\begin{equation}
\label{eq:1D FEM}
 (u_{N}',v'_{N}) + (u_N, v_N) = (f,v^{}_{N})
 \quad \text{ for all } v^{}_{N} \in {V}^{}_{N}(0,1).
\end{equation}
This $u_N^{}$ can be expressed as a linear combination of the $\psi^N_i$
defined in \eqref{eq:1Dbasis}:
\[
u_N^{} = \sum_{j=1, \dots, N-1}\mu_j^{} \psi^N_j(x).
\]
Here the $\mu_j$ are the $N-1$ unknowns  determined by solving the
$N-1$ equations obtained by taking $v_N^{} = \psi_i^N(x)$, for $i=1,
\dots, N-1$,  in \eqref{eq:1D FEM}.
Say we write the  system matrix for these  equations as $(a_2 + a_0)$,
where $a_2$ (usually called the \emph{stiffness matrix}) and $a_0$
(the \emph{mass matrix}) are $(N-1)\times(N-1)$ matrices that correspond,
respectively, to the terms  $(u_{N}',v'_{N})$ and $(u_N, v_N)$. 
Then $a_2$ and $a_0$ are tridiagonal matrices whose stencils are,
respectively, 
\[
N \begin{pmatrix}-1 & 2 & -1
\end{pmatrix}
\quad 
\text{ and }
\quad
\frac{1}{6N} \begin{pmatrix}1 & 4 &1
\end{pmatrix}.
\]

For the two-dimensional problem \eqref{eq:model} there are $(N-1)^2$
unknowns to be determined. Each of these are associated with a node on
the mesh, which we must number uniquely. 
We will follow a standard convention, and use lexicographic
ordering. That is, the node at $(x_i, y_j)$ is labelled
$k=i+(N-1)(j-1)$. The basis function associated with this node is 
$\phi_k^N = \psi_i^N(x)\psi_j^N(x)$.
Then the  $(N-1)^2$ equations in the  linear system for
\eqref{eq:Galerkin} 
can be written as
\[
B(u^{}_{N,N}, \phi_k^N) = (f, \phi_k^N), 
\quad \text{ for } \quad 
k=1, \dots (N-1)^2.
\]
%If $r$ in \eqref{eq:model} were variable, then each entry of the
%system matrix, $A$, would have to be computed using a suitable quadrature
%rule (for a general implementation for arbitrary $r$). %However, for
%the special case of interest here, where $r=1$, we can express it in terms of
Since $r=1$, this can be expressed in terms of 
Kronecker products 
(see, e.g., \cite[Chap. 13]{Laub05}) 
of the one-dimensional matrices described above:
\begin{equation}
\label{eq:kron}
A = a_0 \otimes a_2  
   + a_2 \otimes a_0
   + a_0 \otimes a_0.
\end{equation}
The MATLAB
code for constructing this matrix is found in 
\href{http://www.maths.nuigalway.ie/~niall/SparseGrids/FEM_System_Matrix.m}{\texttt{FEM\_System\_Matrix.m}}.
The main computational cost of executing that function is incurred by
computing the Kronecker products. Therefore, to improve efficiency,   
\eqref{eq:kron} is coded as 
%\begin{quote}
\listbox{.56}
\begin{lstlisting}
A = kron(a0,a2) + kron(a2+a0,a0);
\end{lstlisting}
%\end{quote}

The right-hand side of the finite element linear system is computed by
the function
\href{http://www.maths.nuigalway.ie/~niall/SparseGrids/FEM_RHS.m}{\texttt{FEM\_RHS.m}}.
As is well understood, the order of accuracy of this  quadrature must
be at least that of the underlying scheme, in order for quadrature
errors not to pollute the FEM solution. As given, the code
uses a two-point Gaussian quadrature rule (in each direction) to compute 
\begin{equation}
\label{eq:Galerkin RHS}
b_k = (f,\phi_k^N), \quad \text{ for } \quad 
k=1, \dots, (N-1)^2.
\end{equation}
It can be easily adapted to use a different quadrature scheme (see, e.g., 
\cite[\S 4.5]{GGRS07} for higher order Gauss-Legendre and
Gauss-Lobatto rules for  these
elements).

To compute the error, we need to calculate 
$\sqrt{B(u-u_{N,N}, u-u_{N,N})}$.
Using  Galerkin orthogonality, 
since $u$ solves \eqref{eq:Bilinear} and
$u_{N,N}^{}$ solves \eqref{eq:Galerkin}, one  sees that
\begin{equation}
\label{eq:error trick}
B(u-u_{N,N}, u-u_{N,N}) = (f,u) - (f,u_{N,N}^{}).
\end{equation}
The first term on the right-hand side can be computed 
(up to machine precision) 
using Chebfun 
\listbox{.65}
\begin{lstlisting}
Energy = integral2(u.*f,[0, 1, 0, 1]); 
\end{lstlisting}
where $u$ and $f$ are chebfuns that represent the solution and
right-hand side in \eqref{eq:model}. We estimate the term
$(f,u_{N,N}^{})$ as 
\listbox{.45}
\begin{lstlisting}
Galerkin_Energy =  uN'*b;
\end{lstlisting}
where {uN} is the solution vector, and ${b}$ is the right-hand side of
the linear system, as defined in \eqref{eq:Galerkin RHS}.

\subsection{Numerical results for the Galerkin FEM}
\label{sec:Galerkin numerics}

The test harness for the code described  in Section
\ref{sec:Galerkin implementation} is named 
 \href{http://www.maths.nuigalway.ie/~niall/SparseGrids/Test_FEM.m}{\texttt{Test\_FEM.m}}.
As given, it implements the method for $N=2^4, 2^5, \dots, 2^{10}$, so
the problem size does not exceed what may be  solved on a
standard desktop computer. Indeed, on a computer with at least 8Gb of
RAM, it should run with $N=2^{11}$ intervals in each coordinate
direction, resolving a total of $4,190,209$ degrees of freedom. We
present the results in \autoref{fig:FEM error}, 
%\autoref{tab:Galerkin errors}, 
which  were generated on a computer equipped with  
 32Gb of RAM, allowing  to solve problems with $N=2^{12}$
intervals in each direction (i.e., $16,769,025$ degrees of freedom).
One can observe that the
method is first-order convergent, which is in agreement with 
\eqref{eq:enormGalerkinFEM}. 
\autoref{fig:FEM error 64} shows $u-u_{N,N}^{}$ for $N=64$.

\begin{figure}[h!]
\centering
\begin{subfigure}[h]{0.49\textwidth}
\centering
\includegraphics[scale=.35]{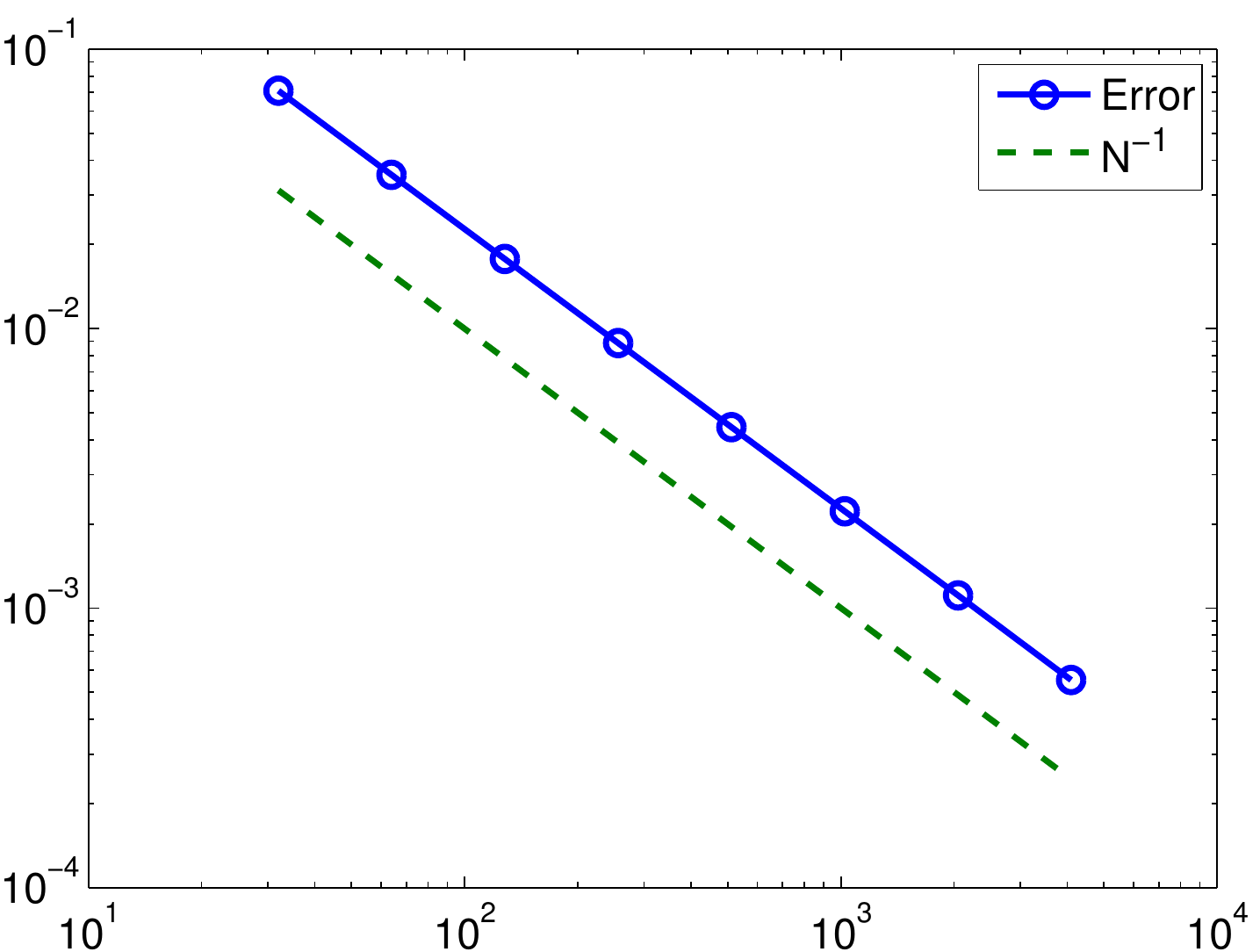}%
\caption{Rate of convergence}
\label{fig:FEM ROC}
\end{subfigure}
\begin{subfigure}[h]{0.49\textwidth}
\centering
\includegraphics[scale=.35]{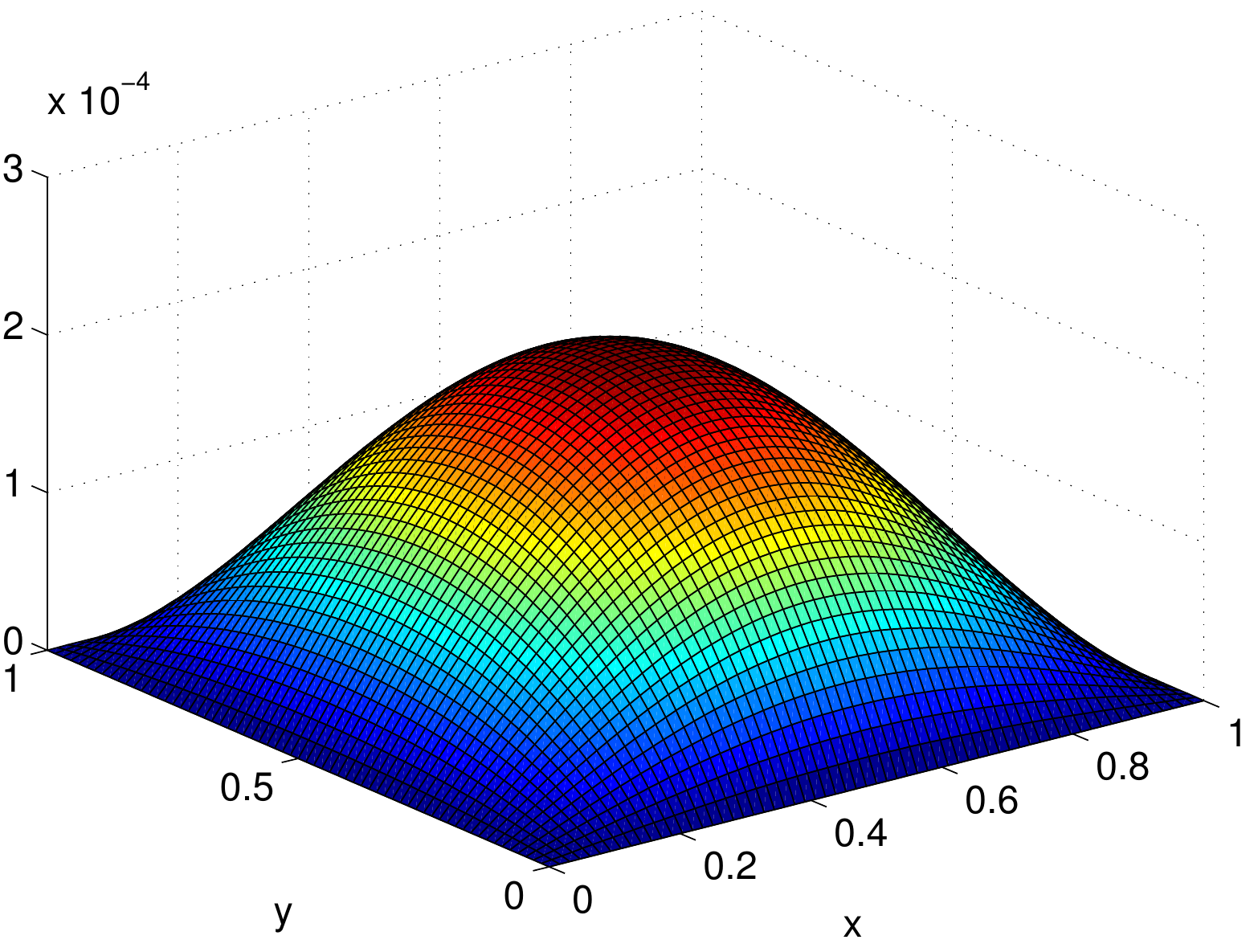}
\caption{$u-u_{64,64}^{}$}
\label{fig:FEM error 64}
\end{subfigure}
\caption{Left: convergence of the classical Galerkin method. Right:
  the error in the Galerkin  solutions with $N=64$}
\label{fig:FEM error}
\end{figure}

\section[two-scale sparse grid FEM]{A two-scale sparse grid finite element method}
\label{sec:two-scale}

The main purpose of this article is to introduce sparse grid FEMs, which can achieve accuracy that is comparable to the
classical Galerkin FEM of Section \ref{sec:Galerkin}, but with much
fewer degrees of freedom.
As we describe  in Section \ref{sec:multiscale}, most of these methods are
described as multiscale or ``multi-level''. However, the simplest
sparse grid method is, arguably, the two-scale method proposed by Lin
et al. \cite{LiYa01}. Those authors were motivated to study the method
in order to prove certain superconvergence results.
For us, its appeal is the simplicity of its
implementation and analysis. By extending our MATLAB program from Section
\ref{sec:Galerkin implementation} by only a few lines of code, we can
obtain a solution that has, essentially, the same accuracy as the
classical FEM, but using  $N^{4/3}$ rather than $N^2$ degrees of
freedom. Moreover, having established some basic principles of the
method in this simple setting, we are well equipped to consider the
more complicated multiscale method of Section~\ref{sec:multiscale}.

Recalling Section \ref{sec:analysis}, the error analysis of a FEM  follows
directly from establishing the approximation properties of the
finite dimensional space and, typically, this follows from an analysis
of a particular candidate for an approximation of the true solution: the
nodal interpolant. Therefore, 
in Section \ref{sec:two-scale interp} we describe a two-scale
interpolation operator. This 
naturally leads to a two-scale FEM, the implementation of which is
described in Section \ref{sec:two-scale implementation}. That section
also contains numerical results that  allow us to
verify that the accuracy of the solution is very similar to the
FEM of Section \ref{sec:Galerkin}. Although the whole
motivation of the method is to reduce the computational cost of
implementing a classical Galerkin method, we postpone a discussion of
the efficiency of the method to Section \ref{sec:good times}, where we
compare the classical, two-scale and multiscale methods directly.

We will make repeated use of the following one-dimensional
interpolation bounds. Their proofs can be found in \cite[Theorem.\ 14.7]{SuMa03}.
\begin{theorem}
 \label{eq:pwlininterpbound}
 Suppose that $u \in H^2(0,1) \cap H^1_0(0,1)$. Then the piecewise
 linear interpolant $I^{}_Nu$ satisfies 
\begin{equation}
 \label{eq:1Dinterpbounds}
  \|u-I^{}_Nu\|_{0,\Omega} \leq Ch_i^2 \|u''\|_{0,\Omega},
\quad \text{  and } \quad
  \|u'-(I^{}_Nu)'\|_{0,\Omega} \leq Ch_i\|u''\|_{0,\Omega},
\end{equation}
for $i=1,2,...,N$, and $h_i = x_i-x_{i-1}$.
\end{theorem}

\subsection{The two-scale interpolant}
\label{sec:two-scale interp}

 Let $V^{}_{N_x}([0,1])$ be the space of piecewise linear functions defined
on the one-dimensional  mesh with  $N_x$ intervals on $[0,1]$.
Define $V^{}_{N_y}([0,1])$ in the same way.
Then let  $V^{}_{N_x,N_y}$ be  the tensor product space
$V^{}_{N_x}([0,1])\times V^{}_{N_y}([0,1])$.
Let $I_{N_x,N_y}:C(\bar{\Omega})\rightarrow V^{}_{N_x,N_y}$ be
the nodal piecewise bilinear
interpolation operator that projects onto  $V^{}_{N_x,N_y}$.
Write $I_{N_x,0}$ for the interpolation operator that
interpolates only in the $x$-direction, so $I_{N_x,0}:C(\bar{\Omega})\rightarrow
V^{}_{N_x}([0,1])\times C([0,1])$. Similarly, let $I_{0,N_y}:C(\bar{\Omega})\rightarrow
C([0,1])\times V^{}_{N_y}([0,1])$ interpolate only in the $y$-direction. Then clearly
\begin{subequations}
\label{eq:InterpID}
 \begin{equation}
  \label{eq:InterpID_a}
  I_{N_x,N_y}= I_{N_x,0} \circ I_{0,N_y} = I_{0,N_y} \circ I_{N_x,0},
 \end{equation}
 \begin{equation}
  \label{eq:InterpID_b}
  \frac{\partial}{\partial x}I_{N_x,N_y}= I_{0,N_y} \circ \frac{\partial}{\partial x}
  I_{N_x,0},
 \end{equation}
 \begin{equation}
  \label{eq:InterpID_c}
  \frac{\partial}{\partial y}I_{N_x,N_y}= I_{N_x,0} \circ \frac{\partial}{\partial y}
  I_{0,N_y}.
 \end{equation}
\end{subequations}

In this section we present an interpretation of the two-scale technique outlined 
in \cite{LMSZ09}. The two-scale interpolation operator
$\widehat{I}^{}_{N,N}:C(\bar{\Omega})\rightarrow V_{N,N}$ is defined as
\begin{equation}
\label{eq:2scaleinterp general}
 \widehat{I}^{}_{N,N}u= I^{}_{N, \sigma(N)} + I^{}_{\sigma(N), N} - I^{}_{\sigma(N), \sigma(N)},
\end{equation}
where $\sigma(N)$ is an integer that divides $N$. 
The following identity appears in \cite{LMSZ09}, and is an integral
component of the following two-scale interpolation analysis:
\begin{equation}
 \label{eq:2scale identity}
I^{}_{N,N}u- \widehat{I}^{}_{N,N}u = (I^{}_{N,0}-I^{}_{\sigma(N),0})(I^{}_{0,N}-I^{}_{0,\sigma(N)}). 
\end{equation}
\begin{theorem}
\label{thm:twoscaleinterp enorm}
 Suppose $\Omega=(0,1)^2$ and $u \in H^1_0(0,1)$. Let $I^{}_{N,N}u$ be the bilinear interpolant of 
 $u$ on $\Omega_{N,N}$, and $\widehat{I}^{}_{N,N}u$ be the two-scale bilinear interpolant
 of $u$ described in~\eqref{eq:2scaleinterp general}. Then there exists a constant $C$ independent 
 of $N$, such that
 \begin{equation*}
  \enorm{\widehat{I}^{}_{N,N}u - I^{}_{N,N}u} \leq C\sigma(N)^{-3}.
 \end{equation*}
\end{theorem}
\begin{proof}
 First from~\eqref{eq:2scale identity}, and then by the first
 inequality in \eqref{eq:1Dinterpbounds}, 
 one has
 \begin{multline}
 \label{eq:twoscaleinterpL2}
 %\begin{split}
  \|\widehat{I}^{}_{N,N}u - I^{}_{N,N}u\|_{0,\Omega} 
  =\|(I^{}_{N,0}-I^{}_{\sigma(N),0})(I^{}_{0,N}-I^{}_{0,\sigma(N)})u\|_{0,\Omega}\\
  \leq C\sigma(N)^{-2}\left\|(I^{}_{0,N}-I^{}_{0,\sigma(N)})\frac{\partial^2 u}{\partial x^2}
  \right\|_{0,\Omega}
  \leq C\sigma(N)^{-2}\sigma(N)^{-2}\left\|\frac{\partial^4 u}{\partial x^2 \partial y^2}
  \right\|_{0,\Omega}\\
  = C\sigma(N)^{-4}\left\|\frac{\partial^4 u}{\partial x^2 \partial y^2}\right\|_{0,\Omega}
  \leq C\sigma(N)^{-4}.
 % \end{split}
 \end{multline}
Following the same reasoning, but this time 
 using the second inequality in \eqref{eq:1Dinterpbounds},
\begin{equation}
\label{eq:grad2scaleinterpL2}
\left\|\frac{\partial}{\partial x}(\widehat{I}^{}_{N,N}u - I^{}_{N,N}u)\right\|_{0,\Omega}
 \leq C\sigma(N)^{-1}\left\|(I^{}_{0,N}-I^{}_{0,\sigma(N)})\frac{\partial^2u}{\partial x^2}
 \right\|_{0,\Omega}
 \leq C\sigma(N)^{-3}.
\end{equation}
Using the same approach, the corresponding bound on
$\|\partial/\partial y(\widehat{I}^{}_{N,N}u - I^{}_{N,N}u)\|_{0,\Omega}$
is obtained, and so
\[\|\nabla(\widehat{I}^{}_{N,N}u - I^{}_{N,N}u)\|_{0,\Omega} \leq C\sigma(N)^{-3}.\]
 To complete the proof, using the definition of the energy norm and the results~\eqref{eq:twoscaleinterpL2}
 and~\eqref{eq:grad2scaleinterpL2} one has
 \begin{multline*}
  \enorm{\widehat{I}^{}_{N,N}u - I^{}_{N,N}u} \leq \|\widehat{I}^{}_{N,N}u - I^{}_{N,N}u\|_{0,\Omega}
  + \|\nabla(\widehat{I}^{}_{N,N}u - I^{}_{N,N}u)\|_{0,\Omega}\\ \leq 
  C\sigma(N)^{-4}+C\sigma(N)^{-3}\leq C\sigma(N)^{-3}.
 \end{multline*}
\end{proof}

This result combined, via  the triangle inequality, with
\autoref{lem:uINN}, leads immediately to the following theorem.
\begin{theorem}
\label{thm:Twoscale u interp enorm}
Let $u$ and $\widehat{I}^{}_{N,N}$ be defined as in Theorem~\ref{thm:twoscaleinterp enorm}.
Then there exists a constant, $C$, independent of
 $N$, such that
\begin{equation*}
 \enorm{u-\widehat{I}^{}_{N,N}u} \leq C(N^{-1} + \sigma(N)^{-3}).
 \end{equation*}
\end{theorem}
\begin{remark}
 \label{rmk:Choice of sigma}
 We wish to choose  $\sigma(N)$ so that the sparse grid
 method is as
 economical as possible while still 
 retaining the accuracy of the classical scheme. 
 Based on the analysis of Theorem~\ref{thm:Twoscale u interp enorm}, 
 we take $\sigma(N)=N^{1/3}$.
\end{remark}

\begin{corollary}
\label{cor:cuberoot}
 Taking $\sigma(N)=N^{1/3}$ in Theorem~\ref{thm:Twoscale u interp enorm}, there exists a constant,
 $C$, such that
 \begin{equation*}
  \enorm{u-\widehat{I}^{}_{N,N}u} \leq CN^{-1}.
 \end{equation*}
\end{corollary}

\subsection{Two-scale sparse grid finite element method}
Let $\psi_i^N(x)$ and $\psi_j^N(y)$ be defined as in~\eqref{eq:1Dbasis}. We now let
$\widehat{V}_{N,N} \subset H^1_0(\Omega)$ be the finite dimensional
space given by
\begin{equation}
\label{eq:two-scale space}
 \widehat{V}_{N,N} = 
 \text{span}\left\{\psi_i^N(x)\psi_j^{\sigma(N)}(y)\right\}^{i=1:N-1}_{j=1:\sigma(N)-1}
 +
 \text{span}\left\{\psi_i^{\sigma(N)}(x)\psi_j^{N}(y)\right\}^{i=1:\sigma(N)-1}_{j=1:N-1}. 
\end{equation}
Now the FEM is: \emph{find $\widehat{u}^{}_{N,N}\in \widehat{V}^{}_{N,N}$ such
  that}
\begin{equation}
 \label{eq:TwoscaleBilinearForm}
 B(\widehat{u}^{}_{N,N},v^{}_{N,N}) = (f,v^{}_{N,N}) 
 \quad \forall v^{}_{N,N} \in \widehat{V}^{}_{N,N}.
\end{equation}
Using the  reasoning that lead to \eqref{eq:enormGalerkinFEM}, 
\autoref{thm:Twoscale u interp enorm} leads to the
following result.
\begin{theorem}
\label{thm:two-scale}
 Let $u$ be the solution to~\eqref{eq:model}, and $\widehat{u}_{N,N}$ the 
 solution to~\eqref{eq:TwoscaleBilinearForm}. Then there exists a constant $C$,
 independent of $N$, such that
$\enorm{u-\widehat{u}_{N,N}} \leq C(N^{-1} + \sigma(N)^{-3})$.
 In particular, taking $\sigma(N)=N^{1/3}$, 
 \begin{equation*}
  \enorm{u-\widehat{u}^{}_{N,N}} \leq CN^{-1}.
 \end{equation*}
\end{theorem}

\subsection{Implementation of the two-scale method}
\label{sec:two-scale implementation}

At first, constructing the linear system for the method
\eqref{eq:TwoscaleBilinearForm} may seem somewhat more daunting than
that for \eqref{eq:Galerkin}. For the classical method
\eqref{eq:Galerkin}, each of the $(N-1)^2$ rows in the 
system matrix has (at most) nine  non-zero entries, because each of the
basis functions shares support with only eight of its neighbours. 
For a general case, where $r$ in \eqref{eq:model} is not constant, and
so \eqref{eq:kron} cannot be used,  the matrix can be computed
either 
\begin{itemize}
\item using a 9-point stencil for each row, which incorporates a
  suitable quadrature rule for the reaction term, or
\item  by iterating over each square in the
mesh, to compute contributions from the four basis functions supported
by that square.
\end{itemize}
In contrast, for any choice of basis for the space 
\eqref{eq:two-scale space}, a single basis function will share support with
$\mathcal{O}(\sigma(N))$ others, so any stencil would be rather
complicated (this is clear from the sparsity pattern of a typical
matrix shown in \autoref{fig:spy2}).
Further, determining the contribution
from the $\mathcal{O}(\sigma(N))$ basis functions that have support on
a single square in a uniform mesh appears to be non-trivial. However,
as we shall see, one can borrow ideas from Multigrid methods to
greatly simplify the process by constructing the linear system from 
entries in the system matrix for the classical method.

We begin by choosing a basis for the space \eqref{eq:two-scale
  space}. This is not quite as simple as taking the union of the sets 
\[
\left\{\psi_i^N(x)\psi_j^{\sigma(N)}(y)\right\}^{i=1:N-1}_{j=1:\sigma(N)-1}
\quad \text{ and } \quad
\left\{\psi_i^{\sigma(N)}(x)\psi_j^{N}(y)\right\}^{i=1:\sigma(N)-1}_{j=1:N-1},
\]
since these two sets are not linearly independent. There are several
reasonable choices of a  basis for the
space. Somewhat arbitrarily, we shall opt for
\begin{equation}
\label{eq:two-scale basis}
\left\{\psi_i^N(x)\psi_j^{\sigma(N)}(y)\right\}^{i=1:N-1}_{j=1:\sigma(N)-1}
\cup 
\left\{\psi_i^{\sigma(N)}(x)\psi_j^{N}(y)
\right\}^{i=1:\sigma(N)-1}_{j=(1:N-1)/(\sigma(N):\sigma(N):N-\sigma(N))}.
\end{equation}
This may be interpreted as taking the union of the usual bilinear
basis functions for an $N \times \sigma(N)$ mesh, and a $\sigma(N)\times
N$ mesh; but from the second of these, we omit any basis functions
associated with nodes found in the first mesh.
For example, if $N=27$, these basis functions can be considered to be
defined on the meshes shown in \autoref{fig:two-scale meshes}, where
each black dot represents the centre of a bilinear basis function that
has support on the  four adjacent rectangles.
\begin{figure}[h!]
\centering
\begin{subfigure}[h]{0.49\textwidth}
\centering
\includegraphics[scale=0.83]{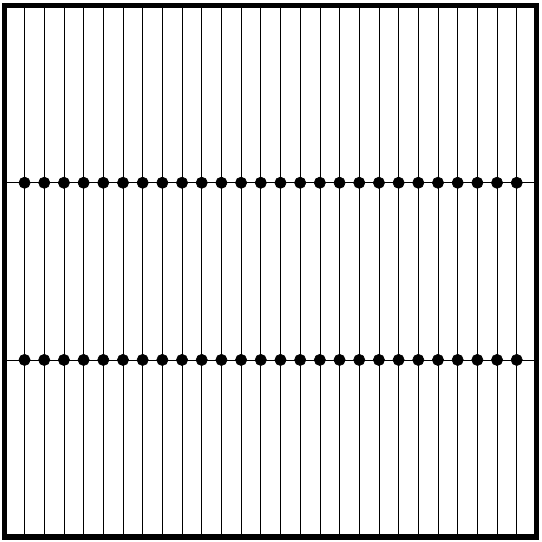}
\caption{A grid for $\left\{\psi_i^{27}(x)\psi_j^{3}(y)\right\}^{i=1:26}_{j=1:2}$}
\label{fig:x-grid}
\end{subfigure}
\begin{subfigure}[h]{0.49\textwidth}
\centering
\includegraphics[scale=0.83]{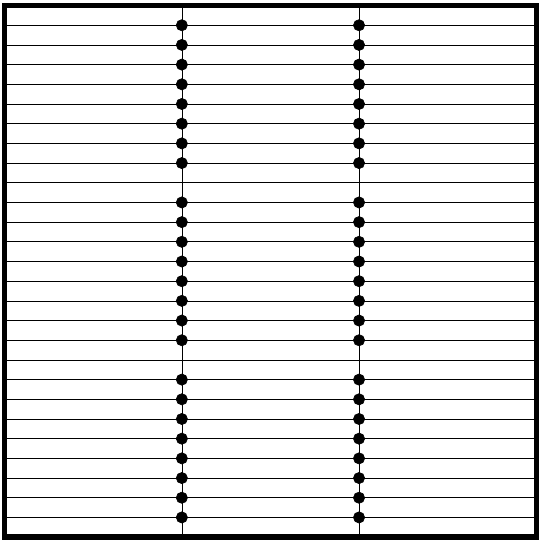}
\caption{A grid for $\left\{\psi_i^{3}(x)\psi_j^{27}(y)\right\}^{i=1:2}_{j=(1:26)/\{9,18\}}$}
\label{fig:y-grid}
\end{subfigure}
\caption{Meshes for the two-scale FEM, for $N=27$}
\label{fig:two-scale meshes}
\end{figure}

To see how to form the linear system associated with the basis
\eqref{eq:two-scale basis} from the entries in \eqref{eq:kron}, we
first start with a one-dimensional problem. Let $V_{\sigma(N)}$ be the
space of piecewise linear functions defined on the mesh
$\omega_x^{\sigma(N)}$, and which vanishes at the end-points. So any
$v\in V_{\sigma(N)}$ can be expressed as 
\[
v(x) = \sum_{i=1:\sigma(N)-1} v_i \psi_i^{\sigma(N)}(x).
\]
Suppose we want to project $v$ onto the space $V_N$. That is, we wish
to find coefficients $w_1^{}, \dots, w_{N-1}^{}$ so that we can write the
same $v$ as 
\[
v(x) = \sum_{i=1:N-1} w_i \psi_i^{N}(x).
\]
Clearly, this comes down to finding an expression for the $\psi_i^{\sigma(N)}$
in terms of the $\psi_i^{N}$. Treating the coefficients
$(v_1,\dots,v_{\sigma(N)})$ and 
$(w_1,\dots,w_{N})$ as vectors, the projection between the
corresponding spaces can be expressed as a $(N-1)\times(\sigma(N)-1)$
matrix. This matrix can be constructed in one line in MATLAB:
\listbox{.74}
\begin{lstlisting}
p = sparse(interp1(x2, eye(length(x2)), x));
\end{lstlisting}
where $x$ is a uniform mesh on $[0,1]$ with $N$ intervals, and 
$x_2$ is a uniform mesh on $[0,1]$ with $N_2$ intervals, where $N_2=\sigma(N)$
is a proper divisor of $N$. 

To extend this to two-dimensions, we form a matrix $P_1$ that projects
a bilinear function expressed in terms of the basis functions in 
\[
\left\{\psi_i^N(x)\psi_j^{\sigma(N)}(y)\right\}^{i=1:N-1}_{j=1:\sigma(N)-1}
\]
to one in $V_{N,N}$ using 
\listbox{.7}
\begin{lstlisting}
P1 = kron(p(2:end-1,2:end-1), speye(N-1));
\end{lstlisting}
Next we construct the matrix $P_2$ that projects
a bilinear function expressed in terms of the basis functions in 
\[
\left\{\psi_i^{\sigma(N)}(x)\psi_j^{N}(y)\right
\}^{i=1:\sigma(N)-1}_{j=(1:N-1)/(\sigma(N):\sigma(N):N-\sigma(N))},
\]
to one in $V_{N,N}$. Part of this process involves identifying the
nodes in $\omega_x^{N}$ which are not contained in
$\omega_x^{\sigma(N)}$. Therefore, we use two lines of MATLAB code to
form this projector:
\listbox{.95}
\begin{lstlisting}
UniqueNodes=sparse(setdiff(1:(N-1), N/N2:N/N2:N-N/N2));
P2 = kron(sparse(UniqueNodes,1:length(UniqueNodes),1),...
    p(2:end-1,2:end-1));
\end{lstlisting}
The actual projector we are looking for is now formed by concatenating
the arrays $P_1$ and $P_2$. That is, we set $P=(P_1|P_2)$. For more
details, see the MATLAB function 
\href{http://www.maths.nuigalway.ie/~niall/SparseGrids/TwoScale_Projector.m}%
{\texttt{TwoScale\_Projector.m}}. 
In a Multigrid setting, $P$ would be referred to as an
\emph{interpolation} or \emph{prolongation} operator, and $P^T$ is
known as a \emph{restriction} operator; see, e.g., \cite{BrHe00}.
It should be noted that, although
our simple construction of the system matrix for the classical
Galerkin method relies on the coefficients in the right-hand side of
\eqref{eq:model}
 being constant, the approach for generating $P$ works in the general case of 
variable coefficients. Also, the use of the MATLAB \texttt{interp1}
function means that modifying the code for a non-uniform mesh is
trivial.

Equipped with the matrix $P$, if the linear system for the %classical
Galerkin method is $A u_{N,N} = b$, then the linear system for the
two-scale method is
$(P^T A P)\widehat{u}_{N,N} = P^T b$. 
The sparsity pattern of $P^TAP$ is shown in \autoref{fig:spy2}.
The solution
can  be projected back onto the original space $V_{N,N}$ by
evaluating $P\widehat{u}_{N,N}$.

\begin{figure}[!h]
 \begin{center}
\includegraphics[width=4.5cm]{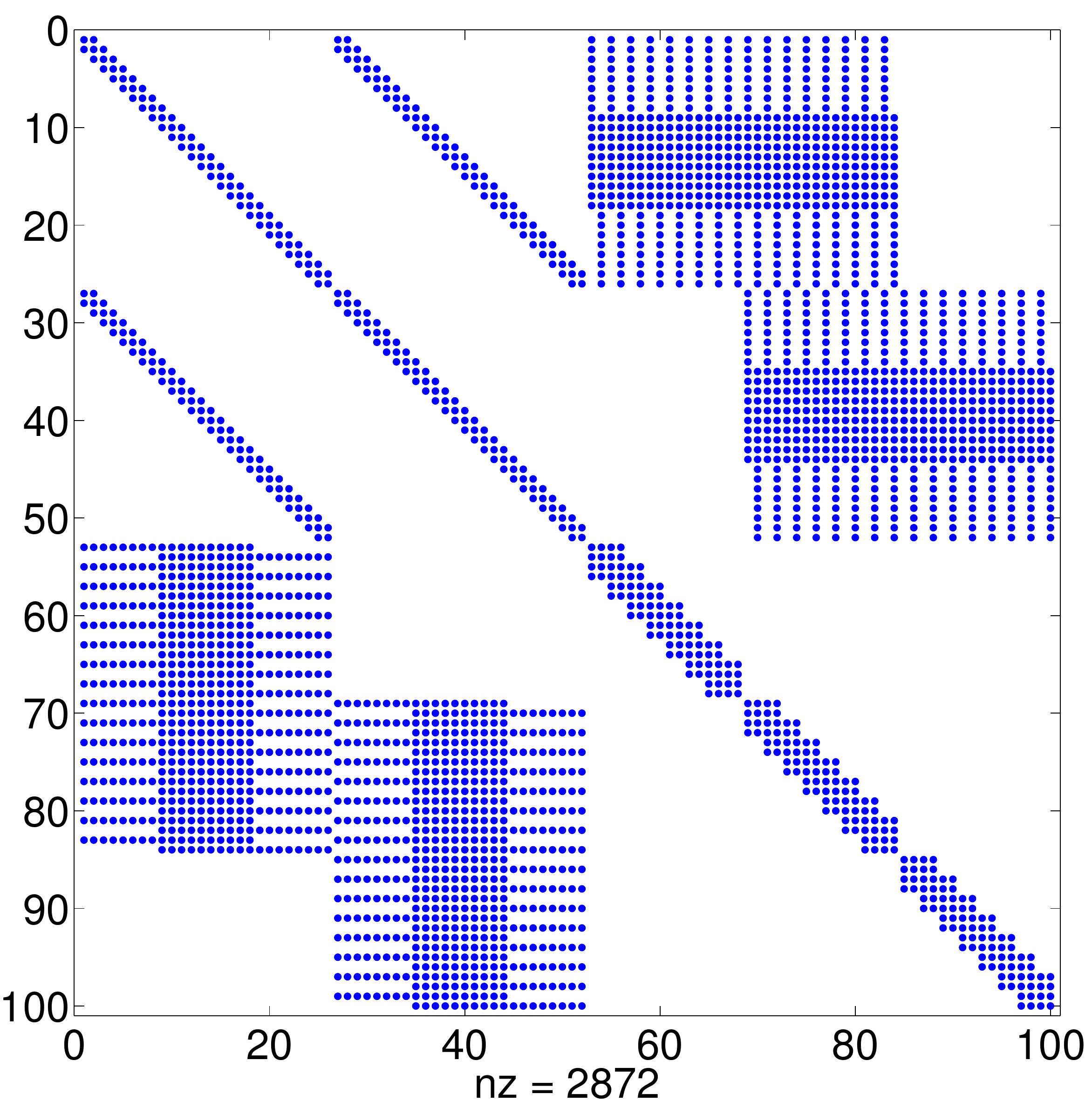} ~
\includegraphics[width=4.5cm]{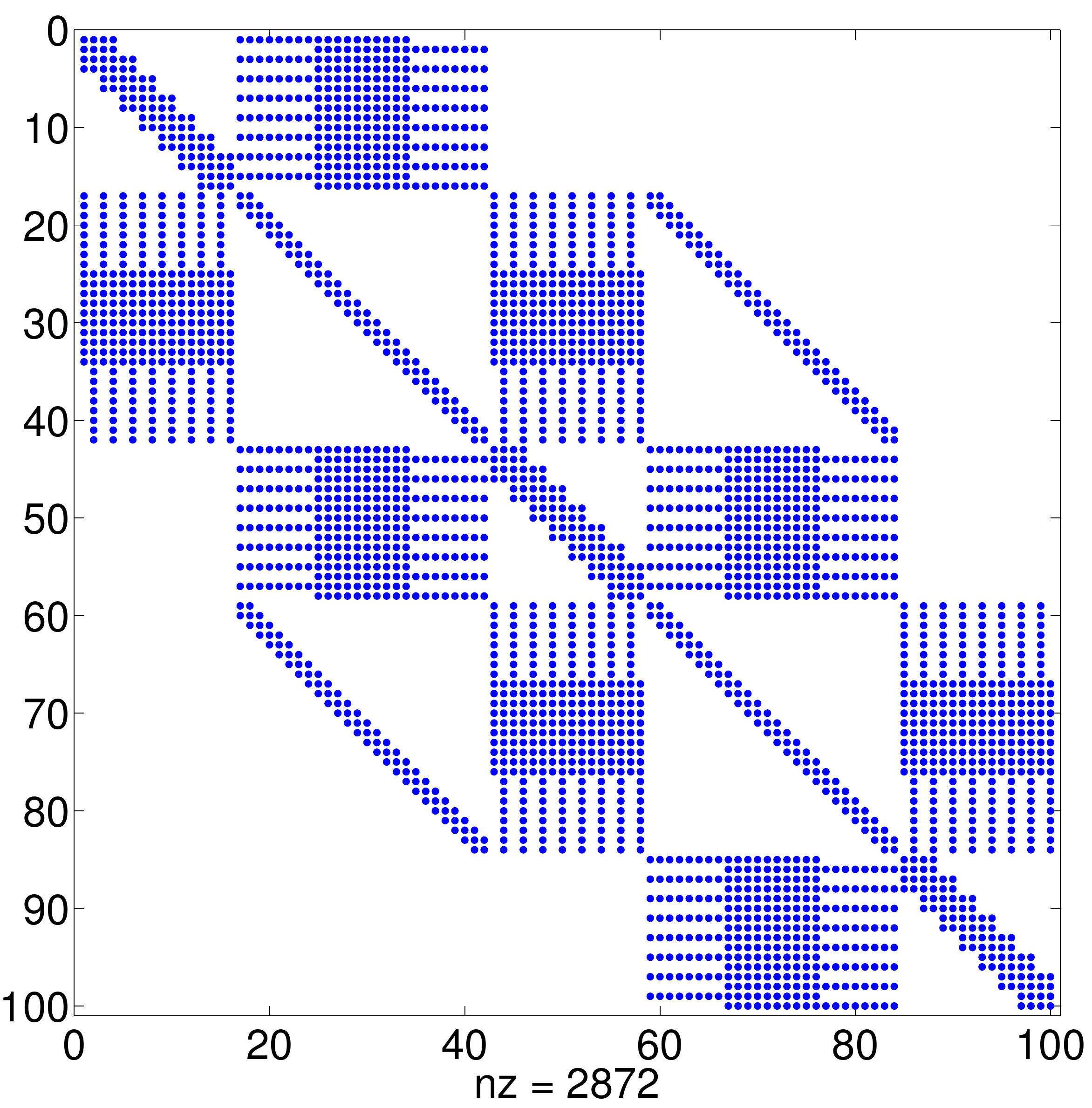}
 \caption{Sparsity patterns  for the two-scale
   method with $N=27$. Left: ordered as in \eqref{eq:two-scale basis}.
   Right: lexicographical ordering}
\label{fig:spy2}
\end{center}
\end{figure}

To use the test harness,
\href{http://www.maths.nuigalway.ie/~niall/SparseGrids/Test_FEM.m}{\texttt{Test\_FEM.m}},
to implement the two-scale  method for our test problem, set the variable \texttt{Method} to
\texttt{two-scale} on Line 18. 

In  \autoref{fig:TwoScale ROC} 
%\autoref{tab:two-scale errors} 
we present computed errors for
various values of $N$, that are chosen, for simplicity, 
to be perfect cubes, which demonstrates that
%These  don't all correspond to values used in \autoref{tab:Galerkin errors}.
%However, from   \autoref{fig:TwoScale ROC} 
Theorem \ref{thm:two-scale} holds
in practice: the method is indeed first-order convergent in the energy
norm. Moreover, the errors for the two-scale method are very similar
to  those of the classical method (the difference is to the order of
1\%) even though far fewer degrees of freedom are used. For
example, when $N=2^{12}$ the classical FEM has 16,769,025 
degrees of freedom, compared with 122,625 for the two-scale method. 
However, comparing  \autoref{fig:FEM error 64} and
\autoref{fig:TwoScale error 64}, we see that the nature of the
point-wise errors are very different. 
We defer further comparisons
between the methods to Section \ref{sec:good times}, where 
efficiency of the various  methods is discussed in detail.

\begin{figure}[h!]
\centering
\begin{subfigure}[h]{0.49\textwidth}
\centering
\includegraphics[scale=.35]{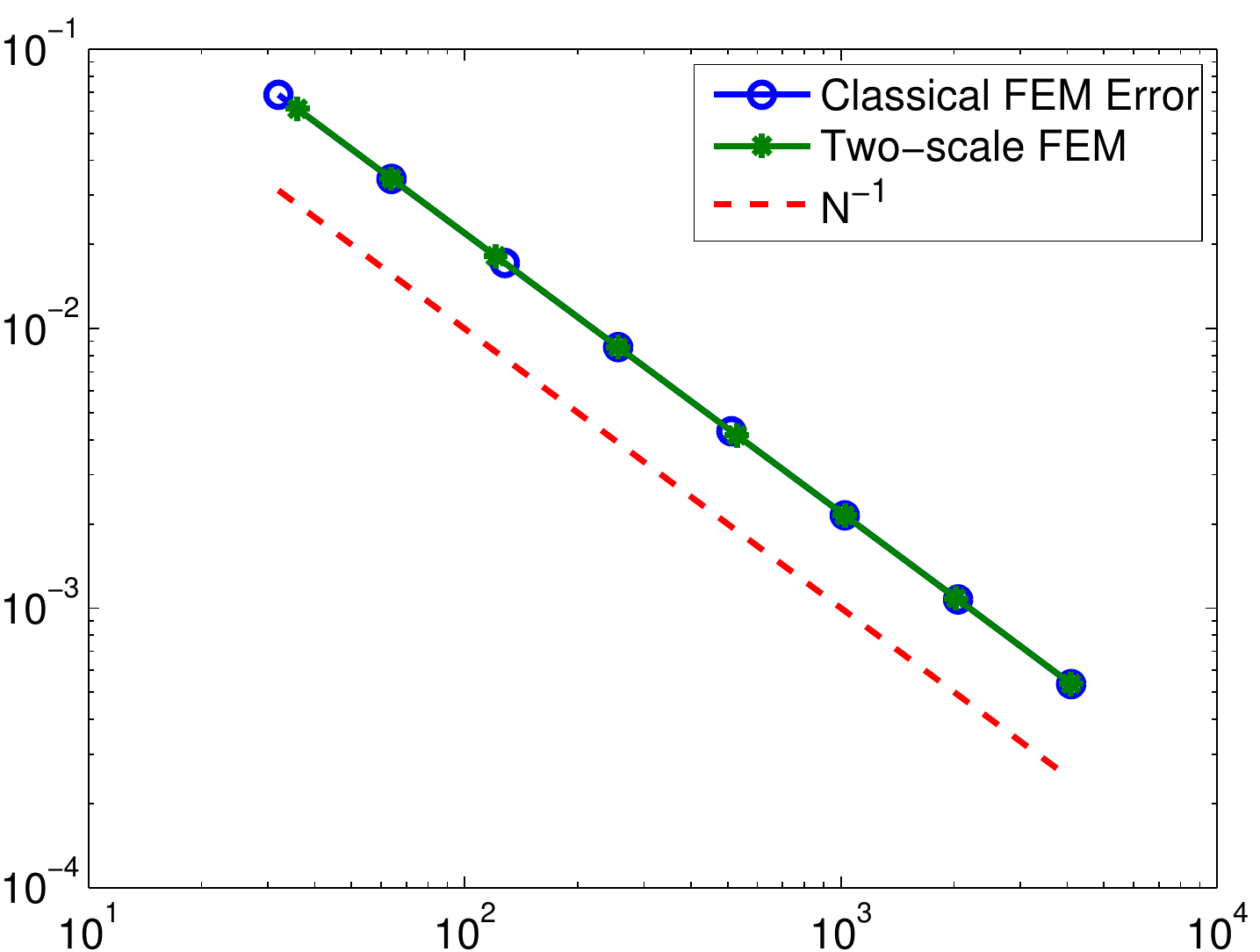}%
\caption{Rate of convergence}
\label{fig:TwoScale ROC}
\end{subfigure}
\begin{subfigure}[h]{0.49\textwidth}
\centering
\includegraphics[scale=.35]{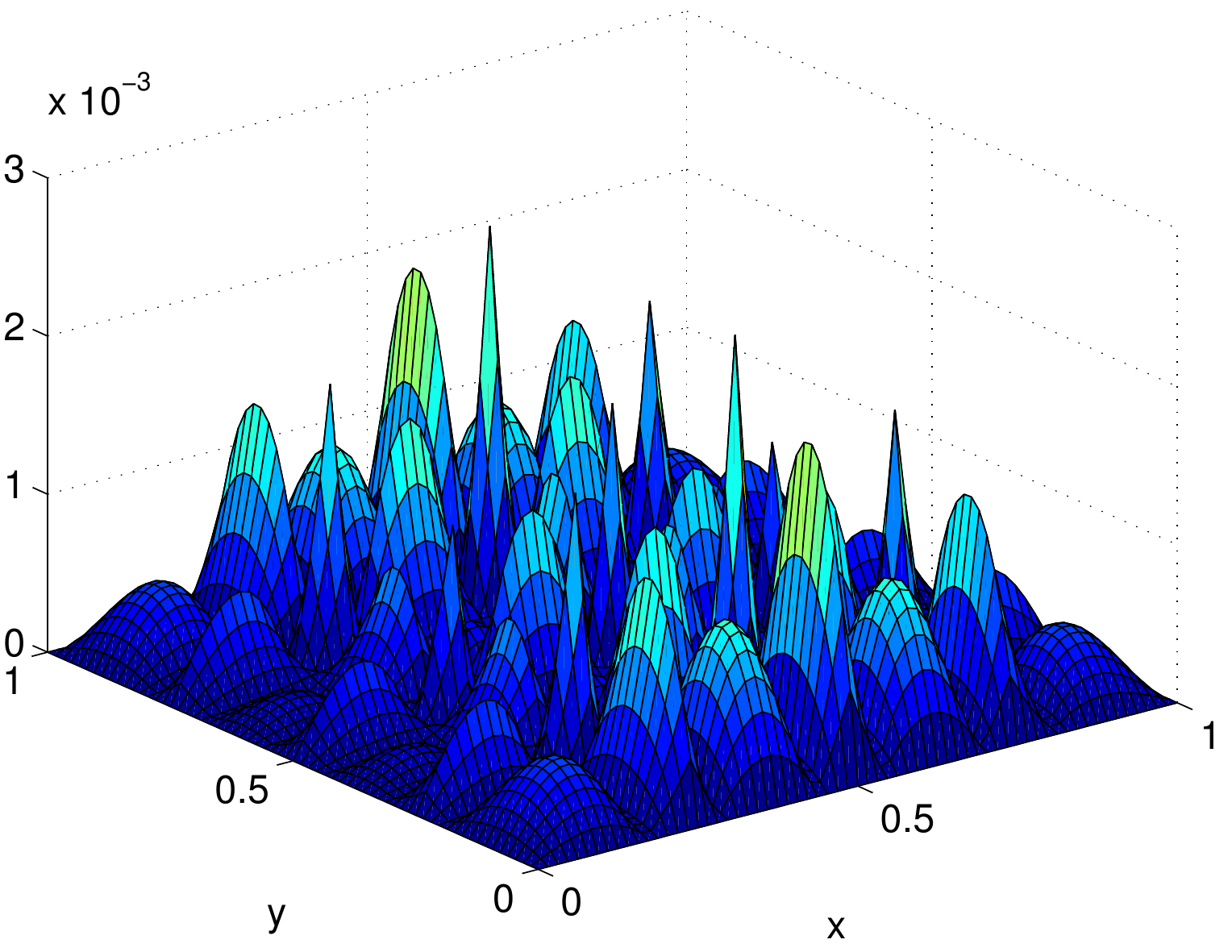}
\caption{$u-\widehat{u}_{64,64}^{}$}
\label{fig:TwoScale error 64}
\end{subfigure}
\caption{Left:  convergence of the classical and two-scale
  methods. Right: error in the  two-scale solution with $N=64$
  (right)}
\label{fig:TwoScale error}
\end{figure}

\section[multiscale sparse grid FEM]{A multiscale sparse grid finite element method}
\label{sec:multiscale}
We have seen that the two-scale method can match the accuracy of the
classical FEM, even though only $\mathcal{O}(N^{4/3})$ degrees of
freedom are used, rather than $\mathcal{O}(N^{2})$. We shall now see
that it is possible to further reduce the required number of degrees
of freedom to $\mathcal{O}(N \log N)$, again without sacrificing the
accuracy of the method very much.
The approach we present is equivalent, up to the choice of basis, to
the hierarchical sparse grid method described by, for example,
Bungartz and Griebel~\cite{BuGr04}. But, because we present it as a
generalisation of the two-scale method, we like to refer to it as the
``multiscale'' method. 

Informally, the idea can be summarised in the following way. Suppose
that $N=2^k$ for some $k$.
For the two-scale method, we solved the problem (in a sense) on two overlapping
grids: one with 
 $N \times N^{1/3}$ intervals, and one with $N^{1/3} \times N$.
Instead, we could apply the same algorithm, but on grids with
  $N \times (N/2)$  and $(N/2) \times  N$ intervals respectively.
Next we apply this same two-scale approach to each of these two grids,
giving three overlapping grids with 
$N \times (N/4)$,  $(N/2) \times (N/2)$,   and $(N/4) \times  N$
intervals.
The process is repeated recursively, until the coarsest grids have $2$
intervals in one coordinate direction --- the smallest feasible
number.

More rigorously, we begin in Section \ref{sec:multiscale interp} by
constructing a multiscale interpolant, using an approach the authors
presented in \cite[\S 3.1]{RuMa14}. This is then analysed in
Section \ref{sec:multiscale analysis}.
As with the two-scale method,
this leads to a FEM, described in \ref{sec:multiscale FEM}.
In Section \ref{sec:multiscale implementation} we
show how this can be programmed, and present numerical results for our
test problem. 

Throughout the analysis, we use the following identities, which are
easily established using, for example, inductive arguments.
\begin{lemma}
 \label{lem:geometric}
 For any $k\geq 1$ we have the following identities
\begin{equation*}
 \begin{split}
  \displaystyle\sum\limits_{i=0}^{k-1}2^i = 2^k-1,\qquad \text{ and } \qquad
  \displaystyle\sum\limits_{i=0}^{k-1}2^{-i} = 2-2^{1-k}.
 \end{split}
\end{equation*}
\end{lemma}

\subsection{The multiscale interpolant}
\label{sec:multiscale interp}
Let $I^{}_{N,N}$ denote the piecewise bilinear interpolation operator on
$V^{}_{N,N}$. Consider the following two-scale
interpolation technique~\eqref{eq:2scaleinterp general}:
\begin{equation}
\label{eq:two scale}
  I^{(1)}_{N,N}= I^{}_{N,\sigma(N)}+I^{}_{\sigma(N),N} - I^{}_{\sigma(N),\sigma(N)},
\end{equation}
where $\sigma(N)$ is an integer that divides $N$. In particular, in
Corollary~\ref{cor:cuberoot}, $\sigma(N)=CN^{1/3}$ is presented as a suitable choice.
However, the same approach can be applied to, say, the terms
$I^{}_{N,\sigma(N)}$ and $I^{}_{\sigma(N),N}$, and again recursively
to the terms that emerge from that.
Following the  approach that is standard for multiscale sparse grid
methods, we  let $\sigma(N)= N/2$.

Let $I_{N_x,N_y}$ denote the piecewise bilinear interpolation operator on
$V^{}_{N_x,N_y}$. The corresponding  Level 1  operator is
\begin{subequations}
\begin{equation}
\label{eq:level 1}
  I^{(1)}_{N_x,N_y}:= I^{}_{N_x,\frac{N_y}{2}}+I^{}_{\frac{N_x}{2},N_y} -
  I^{}_{\frac{N_x}{2},\frac{N_y}{2}}.
\end{equation}
%So if $N_x=N_y$, this is just the same as the two-scale operator of
%\eqref{eq:two scale}, with $\sigma(N)=N/2$. 
The positively signed terms of~\eqref{eq:level 1} are associated with
spaces of dimension $\Oh(N^2/2)$ while the negatively signed term is associated with 
a space of dimension $\Oh(N^2/4)$.
The Level 2  operator is obtained by applying the Level 1
operator to the positive terms in \eqref{eq:level 1}, giving
\begin{equation}
\label{eq:level 2}
  I^{(2)}_{N,N} =
 I^{(1)}_{N,\frac{N}{2}}+I^{(1)}_{\frac{N}{2},N} -
 I^{}_{\frac{N}{2},\frac{N}{2}}
=
 I^{}_{N,\frac{N}{4}}+ I^{}_{\frac{N}{2},\frac{N}{2}} +I^{}_{\frac{N}{4},N}
 - I^{}_{\frac{N}{2},\frac{N}{4}} - I^{}_{\frac{N}{4},\frac{N}{2}}.
\end{equation}
\end{subequations}
Note that the right-hand side of this expression features three
(positively signed) operators
that map to subspaces of dimension $\Oh(N^2/4)$, and two (negatively
signed) operators that map to subspaces of dimension $\Oh(N^2/8)$.
To obtain the  Level 3 operator, we again apply the Level 1 operator
to the positive terms in \eqref{eq:level 2}, because they are the ones
associated with the larger spaces.
In general, the Level $k$ operator is obtained by applying the Level 1
operator of \eqref{eq:level 1} to the positive terms in
$I^{(k-1)}_{N, N}$. This leads to the following definition.

\begin{definition}[Multiscale Interpolation Operator]
\label{def:multiscale}
Let $I^{(0)}_{N,N}= I^{}_{N,N}$ and, from \eqref{eq:level 1},  let
$ I^{(1)}_{N,N}= I^{}_{N,\frac{N}{2}} + I^{}_{\frac{N}{2},N} - I^{}_{\frac{N}{2},\frac{N}{2}}$.
For  $k=2, 3, \dots$, let $I^{(k)}_{N,N}$ be  obtained by
applying the Level 1  operator in \eqref{eq:level 1} to the positively
signed terms in $I^{(k-1)}_{N,N}$.
\end{definition}

We now provide an explicit formula for $I^{(k)}_{N,N}$. This recovers a
standard expression used, for example, in the
 combination technique outlined in~\cite{PfZh99}.
\begin{lemma}
\label{lem:recursive}
Let $I^{(k)}_{N,N}$ be the multiscale  interpolation
  operator constructed in  Definition \ref{def:multiscale} above. Then
 \begin{align}
 \label{eq:multiscaleinterp}
 I^{(k)}_{N,N} = \displaystyle\sum\limits_{i=0}^{k}I^{}_{\frac{N}{2^i},\frac{N}{2^{k-i}}}
 - \displaystyle\sum\limits_{i=1}^{k}I^{}_{\frac{N}{2^i}, \frac{N}{2^{k+1-i}}},
 \qquad \text{ for  } k =0,1,2,\dots.
\end{align}
\end{lemma}
\begin{proof}
It is easy to check that the formula \eqref{eq:multiscaleinterp} is
consistent with the construction given in Definition \ref{def:multiscale}  for $k=0$ and $k=1$.

Next, assume that (\ref{eq:multiscaleinterp}) holds for an
arbitrary  $k=n$; that is,
\begin{equation}
\label{eq:multiinductn}
 I^{(n)}_{N,N} = \displaystyle\sum\limits_{i=0}^{n}I^{}_{\frac{N}{2^i},\frac{N}{2^{n-i}}}
 - \displaystyle\sum\limits_{i=1}^{n}I^{}_{\frac{N}{2^i}, \frac{N}{2^{n+1-i}}}.
 \end{equation}
Following Definition \ref{def:multiscale}, apply the Level 1 operator
\eqref{eq:level 1} to each term in the first sum of the
right-hand side of~(\ref{eq:multiinductn}). This gives
\begin{multline*}
I_{N,N}^{(n+1)}=
 \displaystyle\sum\limits_{i=0}^n \left[ I^{}_{\frac{N}{2^{i}},\frac{N}{2^{n+1-i}}} +
 I^{}_{\frac{N}{2^{i+1}},\frac{N}{2^{n-i}}} - I^{}_{\frac{N}{2^{i+1}},\frac{N}{2^{n+1-i}}} \right]
 - \displaystyle\sum\limits_{i=1}^n
 I^{}_{\frac{N}{2^i},\frac{N}{2^{n+1-i}}}\\
 =  I^{}_{\frac{N}{2^{n+1}},N} + \displaystyle\sum\limits_{i=0}^n \left[
 I^{}_{\frac{N}{2^i},\frac{N}{2^{n+1-i}}} - I^{}_{\frac{N}{2^{i+1}},\frac{N}{2^{n+1-i}}} \right]\\
 = \displaystyle\sum\limits_{i=0}^{n+1} I^{}_{\frac{N}{2^i},\frac{N}{2^{n+1-i}}}
 - \displaystyle\sum\limits_{i=1}^{n+1} I^{}_{\frac{N}{2^i},\frac{N}{2^{n+2-i}}}.
\end{multline*}
That is, \eqref{eq:multiscaleinterp} holds for  $k=n+1$, as
required.
\end{proof}

Although \eqref{eq:multiscaleinterp} is a succinct representation of
the multiscale interpolation operator, for the purposes of analysis
we are actually interested in the difference between operators at successive levels.
First consider the difference between the interpolation  operators at Levels 0
and 1:
\begin{align*}
 I^{(0)}_{N,N}-I^{(1)}_{N,N}= I^{}_{N,N}-I^{}_{N,\frac{N}{2}}
 -I^{}_{\frac{N}{2},N}+I^{}_{\frac{N}{2},\frac{N}{2}}.
\end{align*}
Recalling \eqref{eq:InterpID_a} and \eqref{eq:2scale identity}, this expression can be written as 
\begin{multline}\label{eq:split}
  I^{(0)}_{N,N}-I^{(1)}_{N,N}= (I^{}_{N,0}\circ I^{}_{0,N}) -
  (I^{}_{N,0}\circ I^{}_{0,\frac{N}{2}}) 
  -(I^{}_{\frac{N}{2},0}\circ I^{}_{0,N})+ (I^{}_{\frac{N}{2},0}\circ I^{}_{0,\frac{N}{2}})\\
 = \left(I^{}_{N,0}-I^{}_{\frac{N}{2},0}\right)\left(I^{}_{0,N} -
 I^{}_{0,\frac{N}{2}}\right),
\end{multline}
an identity used repeatedly in \cite{LMSZ09}.
One can  derive a similar expression 
$I^{(1)}_{N,N}-I^{(2)}_{N,N}$:
\begin{gather*}
 I^{(1)}_{N,N}-I^{(2)}_{N,N}= I^{}_{N,\frac{N}{2}} + I^{}_{\frac{N}{2},N} - I^{}_{\frac{N}{2},\frac{N}{2}}
 -I^{}_{N,\frac{N}{4}}- I^{}_{\frac{N}{2},\frac{N}{2}}- I^{}_{\frac{N}{4},N}
 +I^{}_{\frac{N}{2},\frac{N}{4}}+ I^{}_{\frac{N}{4},\frac{N}{2}}.
\end{gather*}
Invoking~(\ref{eq:InterpID_a}) and simplifying we get
\begin{align*}
 I^{(1)}_{N,N}-I^{(2)}_{N,N}=& (I^{}_{N,0}\circ I^{}_{0,\frac{N}{2}})+ (I^{}_{\frac{N}{2},0}\circ I^{}_{0,N})
 - (I^{}_{\frac{N}{2},0}\circ I^{}_{0,\frac{N}{2}})-(I^{}_{N,0}\circ I^{}_{0,\frac{N}{4}})\\
 &- (I^{}_{\frac{N}{2},0}\circ I^{}_{0,\frac{N}{2}})- (I^{}_{\frac{N}{4},0}\circ I^{}_{0,N})
 + (I^{}_{\frac{N}{2},0}\circ I^{}_{0,\frac{N}{4}}) + (I^{}_{\frac{N}{4},0}\circ I^{}_{0,\frac{N}{2}}),\\
&=
  \left(I^{}_{N,0}- I^{}_{\frac{N}{2},0}\right)\left(I^{}_{0,\frac{N}{2}}-
 I^{}_{0,\frac{N}{4}}\right)
		 + \left(I^{}_{\frac{N}{2},0}-I^{}_{\frac{N}{4},0}\right)\left(I^{}_{0,N}-
		 I^{}_{0,\frac{N}{2}}\right).
\end{align*}

We now give the general form for the representation of the difference
between operators at successive levels in terms of one-dimensional
operators. The resulting identity is  an important tool in our analysis
of the bound on the error for the multiscale
interpolation operator.

\begin{lemma}
\label{lem:multiopk_k1}
Let $I^{(k)}_{N,N}$ be the multiscale interpolation operator defined in~(\ref{eq:multiscaleinterp}). Then,
for $k=0,1,2,\dots$,
\begin{equation}
\label{eq:multiopkk1}
 I^{(k-1)}_{N,N}-I^{(k)}_{N,N} = \displaystyle\sum\limits_{i=0}^{k-1} \left(I^{}_{\frac{N}{2^i},0}
 - I^{}_{\frac{N}{2^{i+1}},0}
 \right)\left(I^{}_{0,\frac{N}{2^{k-1-i}}}- I^{}_{0,\frac{N}{2^{k-i}}}\right).
\end{equation}
\end{lemma}
\begin{proof}
 From the expression for the multiscale operator in~(\ref{eq:multiscaleinterp}) we can write
 \begin{multline*}
  I^{(k-1)}_{N,N}-I^{(k)}_{N,N} = \displaystyle\sum\limits_{i=0}^{k-1}I^{}_{\frac{N}{2^i},\frac{N}{2^{k-1-i}}}
  - \displaystyle\sum\limits_{i=1}^{k-1}I^{}_{\frac{N}{2^i},\frac{N}{2^{k-1}}}
  -\displaystyle\sum\limits_{i=0}^{k}I^{}_{\frac{N}{2^{i}},\frac{N}{2^{k-i}}}
  +\displaystyle\sum\limits_{i=1}^{k}I^{}_{\frac{N}{2^i},\frac{N}{2^{k+1-i}}}\\
  =\displaystyle\sum\limits_{i=0}^{k-1}I^{}_{\frac{N}{2^i},\frac{N}{2^{k-1-i}}}
  -\left(\displaystyle\sum\limits_{i=0}^{k-1}I^{}_{\frac{N}{2^i},\frac{N}{2^{k-i}}}-
  I^{}_{N,\frac{N}{2^k}}\right)
  -\left(\displaystyle\sum\limits_{i=0}^{k-1}I^{}_{\frac{N}{2^i},\frac{N}{2^{k-i}}} +
  I^{}_{\frac{N}{2^k},N}\right)
  + \displaystyle\sum\limits_{i=0}^{k-1}I^{}_{\frac{N}{2^{i+1}},\frac{N}{2^{k-i}}}\\
  =\displaystyle\sum\limits_{i=0}^{k-1}\left(I^{}_{\frac{N}{2^i},\frac{N}{2^{k-1-i}}}
  -I^{}_{\frac{N}{2^i},\frac{N}{2^{k-i}}}
  +I^{}_{\frac{N}{2^{i+1}},\frac{N}{2^{k-i}}}\right)
  -\displaystyle\sum\limits_{i=0}^{k-1}I^{}_{\frac{N}{2^i},\frac{N}{2^{k-i}}}
  +I^{}_{N,\frac{N}{2^k}}-I^{}_{\frac{N}{2^k},N}\\
  =\displaystyle\sum\limits_{i=0}^{k-1}\left(I^{}_{\frac{N}{2^i},\frac{N}{2^{k-1-i}}}
  -I^{}_{\frac{N}{2^i},\frac{N}{2^{k-i}}}
  +I^{}_{\frac{N}{2^{i+1}},\frac{N}{2^{k-i}}}\right)
  -\displaystyle\sum\limits_{i=0}^{k-1}I^{}_{\frac{N}{2^{i+1}},\frac{N}{2^{k-1-i}}}\\
  =\displaystyle\sum\limits_{i=0}^{k-1}\left(I^{}_{\frac{N}{2^i},\frac{N}{2^{k-1-i}}}
  -I^{}_{\frac{N}{2^i},\frac{N}{2^{k-i}}}
  +I^{}_{\frac{N}{2^{i+1}},\frac{N}{2^{k-i}}}
  -I^{}_{\frac{N}{2^{i+1}},\frac{N}{2^{k-1-i}}}\right),
 \end{multline*}
which, recalling \eqref{eq:split}, gives the desired expression.
\end{proof}

\subsection{Analysis of the multiscale interpolation operator}
\label{sec:multiscale analysis}
In this section we provide an analysis for the error incurred by the
multiscale interpolant. Standard finite element analysis
techniques then provide 
a full analysis of the underlying method. We begin by establishing a bound, in the energy norm,
for the difference between interpolants at successive levels.
\begin{theorem}
\label{thm:interpIkk1}
Suppose $\Omega=(0,1)^2$. Let $u\in H^1_0(\Omega)$ and
$I^{(k)}_{N,N}$ be the multi-scale interpolation operator defined in~(\ref{eq:multiscaleinterp}). 
Then there exists a constant, $C$, independent of $N$ and $k$, such that
\begin{equation*}
  \enorm{I^{(k)}_{N,N}u-I^{(k-1)}_{N,N}u} \leq C(k4^{k+1}N^{-4} + 4^{k+1}N^{-3}).
 \end{equation*}
\end{theorem}
%for $1\leq k \leq \log_2(N)-1$.
\begin{proof}
By stating $\|I_{N,N}^{(k)}u-I_{N,N}^{(k-1)}u\|_{0,\Omega}$
in the form of Lemma~\ref{lem:multiopk_k1} and applying 
the triangle inequality, along 
with the first inequality in \eqref{eq:1Dinterpbounds} 
and Lemma~\ref{lem:geometric} we see
\begin{multline}
\label{eq:interpIkk1L2}
 \|I_{N,N}^{(k)}u-I_{N,N}^{(k-1)}u\|_{0,\Omega} = \left\|\displaystyle\sum\limits_{i=0}^{k-1}
 (I^{}_{\frac{N}{2^i},0}-I^{}_{\frac{N}{2^{i+1}},0})
 (I^{}_{0,\frac{N}{2^{k-1-i}}}-I^{}_{0,\frac{N}{2^{k-i}}})u\right\|_{0,\Omega}\\
 \leq C\displaystyle\sum\limits_{i=0}^{k-1}\left(\frac{N}{2^{i+1}}\right)^{-2}
 \left\|\left(I_{0,\frac{N}{2^{k-1-i}}}- I_{0,\frac{N}{2^{k-i}}}\right)
 \frac{\partial^2 u}{\partial x^2}\right\|_{0,\Omega}\\
 \leq C\displaystyle\sum\limits_{i=0}^{k-1}\left(\frac{N}{2^{i+1}}\right)^{-2}
 \left(\frac{N}{2^{k-i}}\right)^{-2}
 \left\|\frac{\partial^4 u}{\partial x^2 \partial y^2}\right\|_{0,\Omega}\\
 \leq C\displaystyle\sum\limits_{i=0}^{k-1} (2^{2i+2})(2^{2k-2i})N^{-4}
 = Ck4^{k+1}N^{-4}.
\end{multline}
Using a similar argument, but with the second inequality in
\eqref{eq:1Dinterpbounds},
we have
\begin{multline}
\label{eq:gradinterpIkk1L2}
 \left\|\frac{\partial}{\partial x} \left(\displaystyle\sum\limits_{i=0}^{k-1}
(I^{}_{\frac{N}{2^i},0}-I^{}_{\frac{N}{2^{i+1}},0})
 (I^{}_{0,\frac{N}{2^{k-1-i}}}-I^{}_{0,\frac{N}{2^{k-i}}})u\right)\right\|_{0,\Omega}\\
\leq C_0\displaystyle\sum\limits_{i=0}^{k-1}\left(\frac{N}{2^{i+1}}\right)^{-1}
 \left\|(I^{}_{0,\frac{N}{2^{k-1-i}}}- I^{}_{0,\frac{N}{2^{k-i}}})
 \frac{\partial^2 u}{\partial x^2}\right\|_{0,\Omega}
 \leq C4^{k+1}N^{-3}.
\end{multline}
The corresponding bound on the $y$-derivatives is obtained in a similar manner.
Combining these results gives
\[
 \|\nabla (I^{(k)}_{N,N}u-I^{(k-1)}_{N,N}u)\|_{0,\Omega}
 \leq C4^{k+1}N^{-3}.
\]
Now following from the definition of the energy norm and results~\eqref{eq:interpIkk1L2}
and~\eqref{eq:gradinterpIkk1L2} one has
\begin{multline*}
  \enorm{I^{(k)}_{N,N}u-I^{(k-1)}_{N,N}u} \leq \|\nabla(I^{(k)}_{N,N}u-I^{(k-1)}_{N,N}u)\|_{0,\Omega}
  + \|I^{(k)}_{N,N}u-I^{(k-1)}_{N,N}u\|_{0,\Omega}\\
  \leq C4^{k+1}N^{-3} + Ck4^{k+1}N^{-4}.
 \end{multline*}
\end{proof}

We now want to show that $\enorm{I^{(k)}_{N,N}u-I^{(k-1)}_{N,N}u}$ is
an upper bound for 
$\enorm{I^{(k-1)}_{N,N}u-I^{}_{N,N}u}$, thus showing that in, order to estimate
$\enorm{I^{}_{N,N}u-I^{(k)}_{N,N}u}$, it is enough to look at $\enorm{I^{(k)}_{N,N}u-I^{(k-1)}_{N,N}u}$.

\begin{lemma}
 \label{lem:enorminductive}
 Let $u$ and $I^{(k)}_{N,N}$ be defined as in Theorem~\ref{thm:interpIkk1}. Then there exists a
constant, $C$, independent of $N$ and $k$, such that
 \begin{equation*}
 \enorm{I^{(k-1)}_{N,N}u-I^{}_{N,N}u} \leq \enorm{I^{(k)}_{N,N}u-I^{(k-1)}_{N,N}u}.
 \end{equation*}
\end{lemma}
\begin{proof}
Taking the result of \autoref{thm:interpIkk1} and by applying an
inductive argument one can  easily deduce that
\begin{equation*}
 \displaystyle\sum\limits_{i=1}^{k-1}i4^{i+1}\leq k4^{k+1} 
\quad \text{ and } \quad
 \displaystyle\sum\limits_{i=1}^{k-1} 4^{i+1} \leq 4^{k+1}.
\end{equation*}
 The result then follows by applying the triangle inequality
 and observing that 
 \begin{equation*}
  \displaystyle\sum\limits_{i=1}^{k-1}
  \enorm{I^{(i)}_{N,N}u-I^{(i-1)}_{N,N}u} \leq 
  \displaystyle\sum\limits_{i=1}^{k-1} (i4^{i+1}N^{-4}+4^{i+1}N^{-3}).
 \end{equation*}

\end{proof}

\begin{lemma}
\label{lem:enormIk_I}
Let $u$ and $I^{(k)}_{N,N}$ be defined as in Theorem~\ref{thm:interpIkk1}. Then there exists a
constant, $C$, independent of $N$ and $k$, such that
 \begin{align*}
 \enorm{I^{(k)}_{N,N}u- I^{}_{N,N}u} \leq C4^{k+1}N^{-3}.
 \end{align*}
 \end{lemma}
 \begin{proof}
 By the triangle inequality and the results of Theorem~\ref{thm:interpIkk1} 
 and Lemma~\ref{lem:enorminductive} one has
 \begin{multline*}
  \enorm{I^{(k)}_{N,N}u- I^{}_{N,N}u} \leq \enorm{I^{(k)}_{N,N}u- I^{(k-1)}_{N,N}u} +
  \enorm{I^{(k-1)}_{N,N}u- I^{}_{N,N}u}\\
  \leq C(k4^{k+1}N^{-4} +4^{k+1}N^{-3}).
 \end{multline*}
The result now follows by noting that, for $k\leq N$ (which will always be the case),
 \begin{equation*}
  4^{k+1}N^{-3}\geq k4^{k+1}N^{-4}.
 \end{equation*}
 \end{proof}

Using the triangle inequality, we combined this result with
\autoref{lem:uINN} and \ref{lem:enormIk_I}  to establish the following
error estimate.
 \begin{theorem}
 \label{thm:enorm_u_I_k}
 Let $u$ and $I^{(k)}_{N,N}$ be defined as in Theorem~\ref{thm:interpIkk1}. Then there exists a
constant, $C$, independent of $N$ and $k$, such that
 \begin{equation*}
  \enorm{u-I^{(k)}_{N,N}u}\leq C(N^{-1}+ 4^{k+1}N^{-3}).
  \end{equation*}
 \end{theorem}
 
 \begin{remark}
 \label{rmk:Choice of k}
  Our primary goal is to construct an efficient finite element method, by taking the smallest
  possible space on which $I^{(k)}_{N,N}u$ can be defined. Thus we want to take $k$ as large as 
  is possible while retaining the accuracy of the underlying method. On a uniform mesh the 
  coarsest grid must have at least two intervals in each coordinate direction. 
  From~\eqref{eq:multiscaleinterp} it is clear that the coarsest grid we interpolate over 
  has $N/2^k$ intervals. Thus the largest and most useful value of $k$ we can choose is 
  $\khat=\log_2N-1$. In light of this, we have the following result.
 \end{remark}

 \begin{corollary}
 \label{cor:multiscale}
 Let $u$ and $I^{(\khat)}_{N,N}$ be defined as in Theorem~\ref{thm:interpIkk1}.
 Taking $k:=\khat=\log_2(N)-1$
in Theorem~\ref{thm:enorm_u_I_k}, there exists a
constant, $C$, independent of $N$, such that
  \begin{equation*}
   \enorm{u-I^{(\khat)}_{N,N}u} \leq CN^{-1}.
  \end{equation*}
 \end{corollary}
 \subsection{Multiscale sparse grid finite element method}
\label{sec:multiscale FEM}
 Let $\psi_i^N$ and $\psi_j^N$ be as defined in~\eqref{eq:1Dbasis}
 We now define the finite dimensional space $V^{(k)}_{N,N} \subset
 H^1_0(\Omega)$ as 
\begin{multline*}
 V^{(k)}_{N,N} =
\mathrm{span}\left\{ \psi^{N}_i(x) \psi^{N/2^{k}}_j(y)\right\}^{i=1:N-1}_{j = 1:N/2^{k}-1}\\
+\mathrm{span}\left\{ \psi^{N/2}_i(x)
  \psi^{N/2^{k-1}}_j(y)\right\}^{i=1:N/2-1}_{j=1:N/2^{k-1}-1}\\ 
+ \cdots
+\mathrm{span}\left\{ \psi^{N/2^{k-1}}_i(x)
  \psi^{N/2}_j(y)\right\}^{i=1:N/2^{k-1}-1}_{j = 1:N/2-1} 
+ \mathrm{span}\left\{ \psi^{N/2^k}_i(x) \psi^{N}_j(y)
\right\}^{i=1:N/2^k-1}_{j = 1:N-1}. 
\end{multline*}
Note that each two dimensional basis function is  the product of
 two one-dimensional functions that are of different scales. This
 description involves
$(k+1)2^{-k}N^2 + (2^{-k+1}-4)N + k+1$
functions in the spanning set, which are
 illustrated in
 the diagrams in \autoref{fig:multiscale basis2}. The left most diagram
 shows a uniform mesh with $N=16$ intervals in each coordinate
 direction. Each node represents a basis function for the space
$V^{(0)}_{N,N} =  V^{}_{N,N}$.

\begin{figure}[htb]
\centering
\includegraphics[scale=.8]{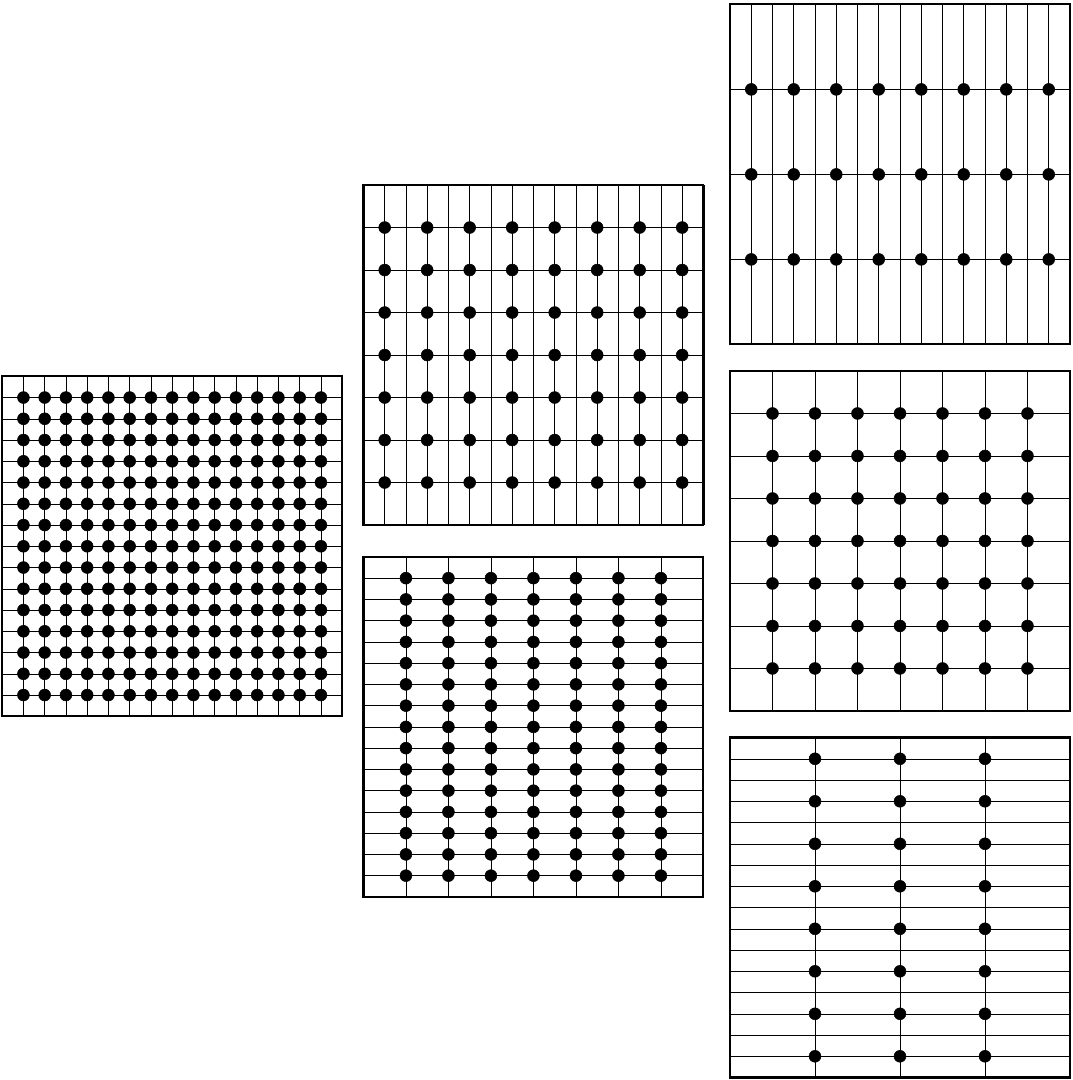}
\caption{Meshes for the multiscale method for $N=16$, based on the
  spaces (left to right) $V^{(0)}_{N,N}$, $V^{(1)}_{N,N}$ and
  $V^{(2)}_{N,N}$}
\label{fig:multiscale basis2}
\end{figure}

In the centre column of \autoref{fig:multiscale basis2} we show the grids associated with
$V^{(1)}_{N,N}$.
Notice that the spaces associated with these two grids are
not linearly independent. To form an invertible system matrix, in
\autoref{fig:multiscale basis2} we
highlight  a particular choice of non-redundant basis functions by
 solid circles:
\begin{align*}
\left\{ \psi^{N}_i(x) \psi^{N/2}_j(y) \right\}^{i=1:2:N-1}_{j = 1:N/2-1}
\quad \text{ and } \quad
\left\{ \psi^{N/2}_i(x) \psi^{N}_j(y) \right\}^{i=1:N/2-1}_{j = 1:N-1}.
\end{align*}

The right-most  column of \autoref{fig:multiscale basis2} show the
grids associated with $V^{(2)}_{N,N}$. Again we use solid
circles to represent our choice of basis:
\begin{multline*}
\left\{ \psi^{N}_i(x) \psi^{N/4}_j(y)\right\}^{i=1:2:N-1}_{j = 1:N/4-1}, \quad
\left\{ \psi^{N/2}_i(x) \psi^{N/2}_j(y) \right\}^{i=1:N/2-1}_{j =
  1:N/2-1},\\ \text{ and } \quad
\left\{ \psi^{N/4}_i(x) \psi^{N}_j(y) \right\}^{i=1:N/4-1}_{j = 1:2:N-1}.
\end{multline*}
There are many ways in which one can choose which basis functions to include. The way we 
have chosen, excluding the boundary, is as follows:
\begin{itemize}
 \item when $N_x>N_y$ include every second basis function in the $x$-direction;
 \item when $N_x=N_y$ or when $2N_x=N_y$ include all basis functions;
 \item from the remaining subspaces include every second basis function in the $y$-direction.
\end{itemize}
In general the choice of basis we make for the space $V^{(k)}_{N,N}$
is dependent on whether $k$ is odd or even.
When $k$ is odd the basis we choose is:
\begin{subequations}
\label{eq:comb basis}
\begin{multline}
\bigcup\limits_{l=0}^{(k-1)/2}\left\{\psi_i^{N/2^l}\psi_j^{N/2^{k-l}}
\right\}^{i=1:2:N/2^{l}-1}
_{j=1:N/2^{k-l}-1}
\bigcup \left\{\psi_i^{N/2^{(k+1)/2}}\psi_j^{N/2^{(k-1)/2}}\right\}^{i=1:N/2^{(k+1)/2}-1}
_{j=1:N/2^{(k-1)/2}-1}\\
\bigcup\limits_{l=(k+3)/2}^{k}\left\{\psi_i^{N/2^l}\psi_j^{N/2^{k-l}}\right\}^{i=1:N/2^{l}-1}
_{j=1:2:N/2^{k-l}-1};
\end{multline}
And when $k$ is even the basis we choose is:
\begin{multline}
\bigcup\limits_{l=0}^{k/2-1}\left\{\psi_i^{N/2^l}\psi_j^{N/2^{k-l}}
\right\}^{i=1:2:N/2^{l}-1}
_{j=1:N/2^{k-l}-1}
\bigcup \left\{\psi_i^{N/2^{k/2}}\psi_j^{N/2^{k/2}}\right\}^{i=1:N/2^{k/2}-1}
_{j=1:N/2^{k/2}-1}\\
\bigcup\limits_{l=k/2+1}^{k}\left\{\psi_i^{N/2^l}\psi_j^{N/2^{k-l}}\right\}^{i=1:N/2^{l}-1}
_{j=1:2:N/2^{k-l}-1}.
\end{multline}
\end{subequations}
This has dimension $2^{-k}(k/2+1)N^2-2N+1$.
For a computer implementation one can, of course, choose alternative
 ways of expressing the basis. Although these are mathematically equivalent, they can lead
to different linear systems.
 We can now
formulate the  multiscale  sparse grid finite element method: \emph{find $u_{N,N}^{(k)} \in {V}_{N,N}^{(k)}$}
such that
\begin{equation}
\label{eq:sparse FEM}
 \mathcal{B}(u_{N,N}^{(k)},v^{}_{N,N}) = (f,v^{}_{N,N})
 \quad \text{ for all } v^{}_{N,N} \in {V}^{(k)}_{N,N}.
\end{equation}
We will consider the analysis for the case $k:=\khat=\log_2N-1$.

\begin{theorem}
\label{thm:sparse FEM error}
Let $u$ be the solution to \eqref{eq:model}, and
  $u_{N,N}^{(\khat)}$ the solution to \eqref{eq:sparse FEM}. Then
 there exists a constant $C$, independent of $N$ and $\Eps$,  such that
 \begin{equation*}
  \enorm{u-u^{(\khat)}_{N,N}} \leq CN^{-1}.
 \end{equation*}
 \end{theorem}
\begin{proof}
The bilinear form \eqref{eq:sparse FEM} is continuous and coercive, so
it  follows from classical finite element analysis that
\[
  \enorm{u-u^{(\khat)}_{N,N}} \leq C \displaystyle\inf_{\psi \in
    V^{(\khat)}_{N,N}(\Omega)}\enorm{u-\psi}.
\]
Since  $I^{(k)}_{N,N}u \in V^{(\khat)}_{N,N}(\Omega)$ the result
is an immediate consequence of
Corollary~\ref{cor:multiscale}.
\end{proof}

%%%%%%%%%%%%%%%%
\subsection{Implementation of the multiscale method}
\label{sec:multiscale implementation}
We now turn to the construction of the linear system for the method~\eqref{eq:sparse FEM}.
As with the construction of the linear system
for~\eqref{eq:TwoscaleBilinearForm}, 
we begin by building matrices that project from the one-dimensional
subspaces of $V_{N}$ to 
the full one-dimensional space $V_N$. In MATLAB this is done with the
following line of code 
\listbox{.77}
\begin{lstlisting}
px = sparse(interp1(Sx, eye(length(Sx)), x));
\end{lstlisting}
where $x$ is a uniform mesh on $[0,1]$ with $N$ intervals, and $S_x$
is a submesh of $x$ 
with $N/M$ intervals and $M$ is a proper divisor of $N$.

Next we build a collection of two-dimensional projection matrices, $P$. 
Each of these projects a bilinear function
expressed in terms of each of the basis functions in~\eqref{eq:comb basis}, to one in $V_{N,N}$.
This is done using 
\href{http://www.maths.nuigalway.ie/~niall/SparseGrids/MultiScale_Projector.m}% 
{\texttt{MultiScale\_Projector.m}}, a function written in MATLAB, and the following piece of code:
\listbox{.87}
\begin{lstlisting}
P = kron(py(2:N,2:skip_y:My), px(2:N,2:skip_x:Mx));
\end{lstlisting}
Here $M_x$ and $M_y$ are the number of intervals in the coarse $x$-
and $y$-directions 
respectively. The values $skip_x$ and $skip_y$ are set to either $1$ or $2$ depending on whether we
include all basis functions, or every other basis function, in~\eqref{eq:comb basis}.
The projectors in the collection are concatenated to give the projector from $V^{(k)}_{N,N}$ to $U^{}_{N,N}$
That is $P=(P_1|P_2|\dots|P_k|P_{k+1})$.

As is the case with the two-scale method, having constructed $P$ and given the linear system for the classical
Galerkin method, $Au_{N,N}=b$, the linear system for the multiscale method is $(P^TAP)u^{(k)}_{N,N}=P^Tb$.
Evaluating $Pu^{(k)}_{N,N}$ then projects the solution back to the original space $V_{N,N}$.

To use the test harness,
\href{http://www.maths.nuigalway.ie/~niall/SparseGrids/Test_FEM.m}{\texttt{Test\_FEM.m}},
to implement the multiscale  method for our test problem, set the variable \texttt{Method} to
\texttt{multiscale} on Line 18. 

In \autoref{fig:MultiScale ROC}
we present numerical results for
the multiscale method  which demonstrate 
that, in practice, the $CN^{-1}$ term from the theoretical results
of Theorem~\ref{thm:sparse FEM error} 
is observed numerically. The errors for the multiscale method are very
similar to those of  
the classical Galerkin method and the two-scale method, however the
degrees of freedom are greatly reduced. 
We saw in Section~\ref{sec:two-scale implementation}, that for
$N=2^{12}$ the classical Galerkin FEM 
involves $16,769,025$ degrees of freedom, compared to $122,625$ for
the two-scale method. 
This is further reduced to $45,057$ degrees of freedom for the
multiscale method (taking $k=\log_2N-1$). 
Again comparing \autoref{fig:FEM error 64}, \autoref{fig:TwoScale error 64}
and \autoref{fig:MultiScale_error64} we see that although the nature of the point-wise errors are quite 
different, they are all similar in magnitude.

\begin{figure}[htb]
\begin{subfigure}[h]{0.49\textwidth}
\centering
\includegraphics[scale=.35]{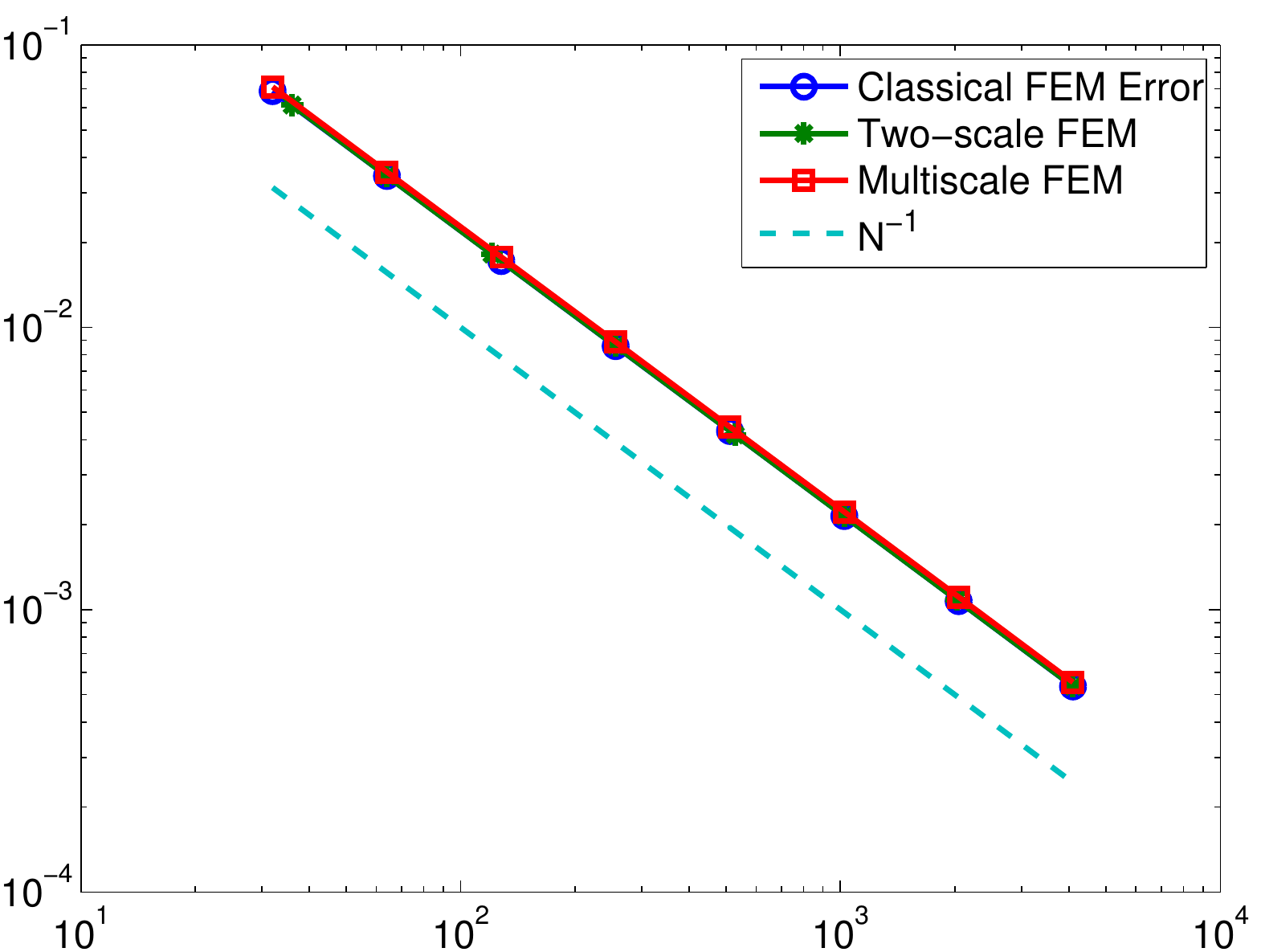}%
\caption{Rate of convergence}
\label{fig:MultiScale ROC}
\end{subfigure}
\begin{subfigure}[h]{0.49\textwidth}
\centering
\includegraphics[scale=.35]{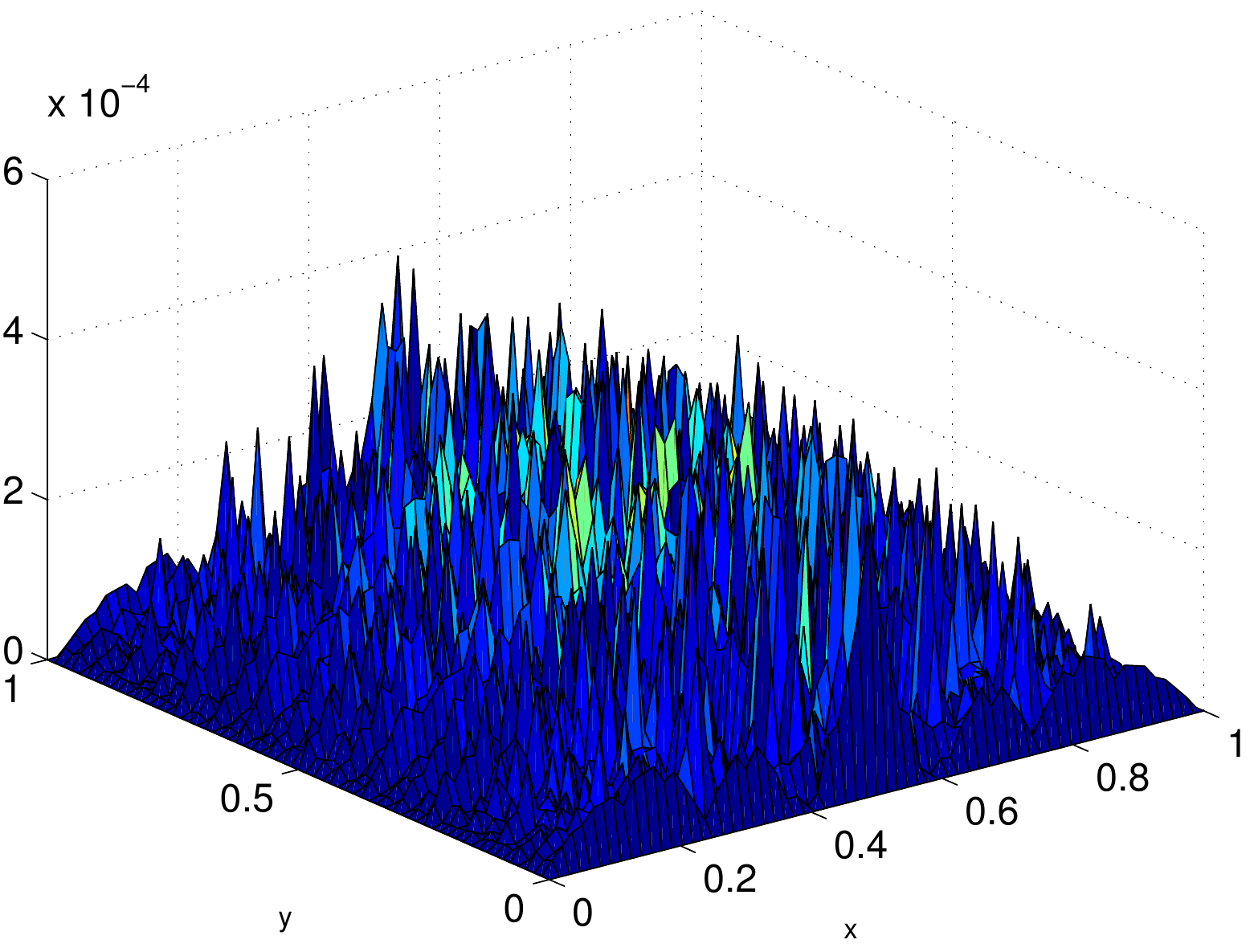}%
\caption{$u-u_{64,64}^{(5)}$}
\label{fig:MultiScale_error64}
\end{subfigure}
\caption{Left: the convergence of the classical, two-scale and multiscale
  methods. Right: error in the multiscale solution with $N=64$ and $k=5$}
\label{fig:multiscale error 64}
\end{figure}

\subsection{Hierarchical basis}
\label{sec:hierarchical}
The standard choice of basis used in the implementation of a sparse
grid method is generally the hierarchical  
basis~(\cite{Yser86}), which is quite different from that given 
in \eqref{eq:comb basis}. The most standard reference for sparse grids is the work
 of Bungartz and  Griebel \cite{BuGr04}, so we adopt their notation.
 
 The hierarchical basis for the (full) space $V^{}_{N,N}(\Omega)$, is 
\begin{equation}
\label{eq:tensor sum full hier}
 \bigoplus_{m,n=1,\dots,\log_2N}W_{m,n},
\quad \text{ where } \quad
 W_{m,n}= \left\{\psi^{2^m}_i\psi^{2^n}_j\right\}^{i=1:2:2^m-1}_{j=1:2:2^n-1}.
\end{equation}

This is shown in Figure~\ref{fig:hier basis full} for the case $N=8$. The 
hierarchical basis is constructed from $9$ subspaces. Recall that each
dot on these 
grids represents a basis function with support on the adjacent four rectangles.
For example, $W_{11}$ is shown in the top left, and represents a
simple basis function 
with support on the whole domain. In the subspace $W_{33}$ in the bottom right of
Figure~\ref{fig:hier basis full} each basis function has support within
the rectangle on which it is centred. No two basis functions in a
given subspace share  
support within that subspace.

\begin{figure}[!h]
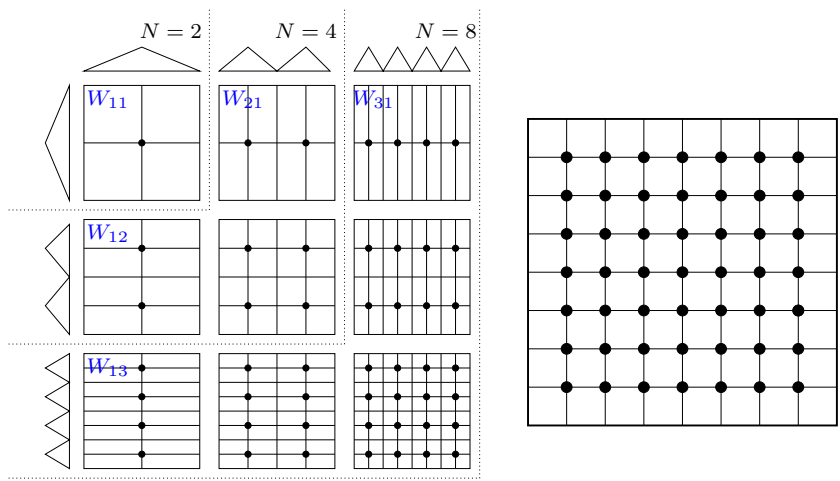

\begin{center}
 \input{hierbasisfulln3.pspdftex}\qquad
 \input{multibasisV_0N8.pspdftex}
 \end{center}
 \caption{Left: All subspaces required to build a hierarchical basis for a full grid 
 where $N=8$. Right: Full grid with basis functions for $N=8$.}
 \label{fig:hier basis full}
\end{figure}

A sparse grid method is constructed by omitting a certain number of the
subspaces $W_{m,n}$ in \eqref{eq:tensor sum full hier}. For a typical
problem, one uses only those subspaces
$W_{m,n}$ for which  $m+n\leq\log_2N+1$: compare Figures 3.3 and 3.5 of
\cite{BuGr04}. This gives the same spaces as defined by \eqref{eq:comb basis}, 
taking $k=\log_2N-1$.
Expressed in terms of
  a hierarchical basis, the basis for $V^{(k)}_{N,N}$ is written as
\begin{equation}
\label{eq:comb hier}
\bigoplus_{\substack{m,n=1,\dots,\log_2N \\ m+n\leq \log_2N+1}}W_{m,n}.
\end{equation}
Both choices give rise to basis functions
centred at the same grid points (but with different support)
and have linear systems of the same size, but different structure: see
\autoref{fig:spymulti} and \autoref{fig:spymultihier}.
%\texttt{enter sparsity diagrams}

 \begin{figure}[!h]
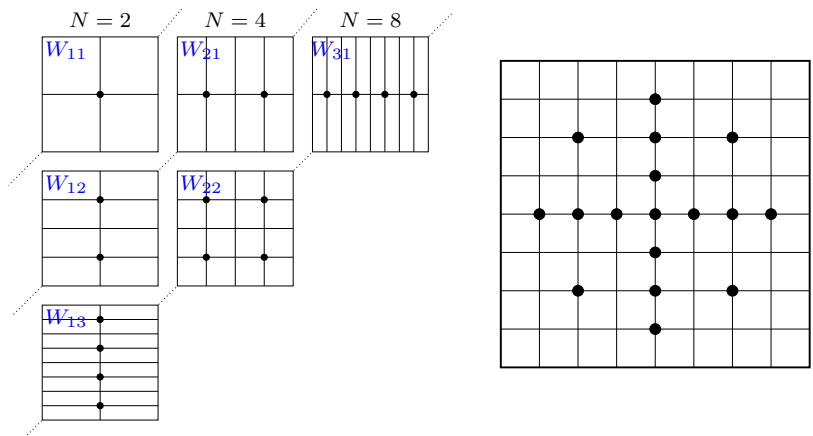

\begin{center}
 \input{hierbasisSparsen3.pspdftex}\qquad
 \input{multibasisSpaV_0N8.pspdftex}
 \end{center}
 \caption{Left: All subspaces required to build a hierarchical basis for a sparse grid 
 where $N=8$. Right: Sparse grid with basis functions for $N=8$.}
 \label{fig:hier basis sparse}
\end{figure}

It is known that this  basis leads to reduced sparsity of the linear system (see, e.g., 
\cite[p. 250]{HGCh07}). Numerical experiments have shown that the
basis given in \eqref{eq:comb basis} leads to fewer non-zeros entries in
the system matrix, and so is our preferred basis. In \autoref{fig:spymulti}
we compare the sparsity pattern of the system matrices for the
basis described in~\eqref{eq:comb basis}
for the case $N=32$ using a natural ordering (left) and lexicographical ordering (right).
For the same value of $N$, \autoref{fig:spymultihier} shows
the sparsity patterns for the system matrices constructed
using the hierarchical basis~\eqref{eq:comb hier}, again using natural
 and lexicographical 
ordering. The matrices in \autoref{fig:spymultihier} contain
more nonzero entries and are much denser than those in \autoref{fig:spymulti}.

\begin{figure}[htb]
 \centering
 \includegraphics[width=4.5cm]{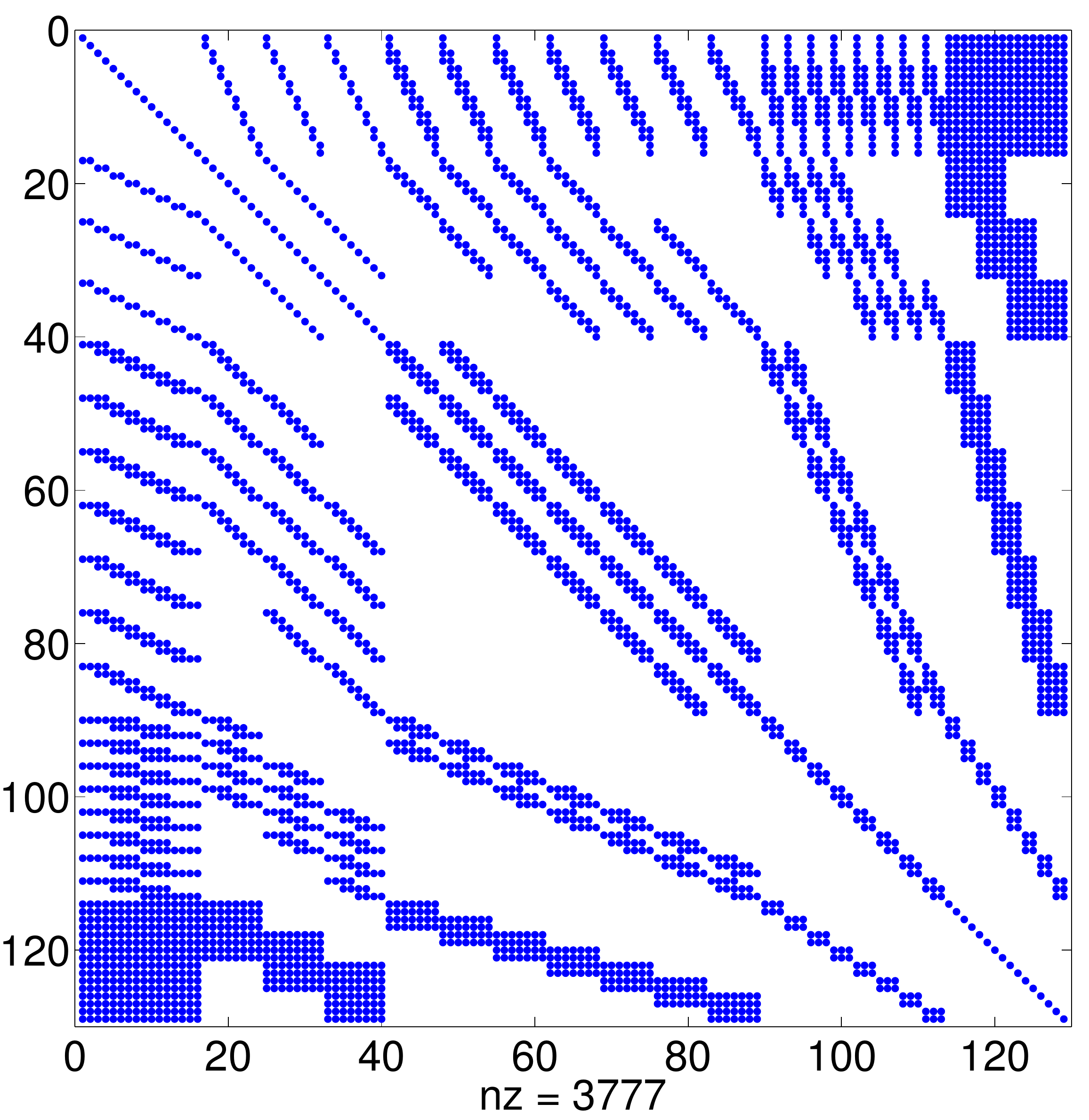}\quad
 \includegraphics[width=4.5cm]{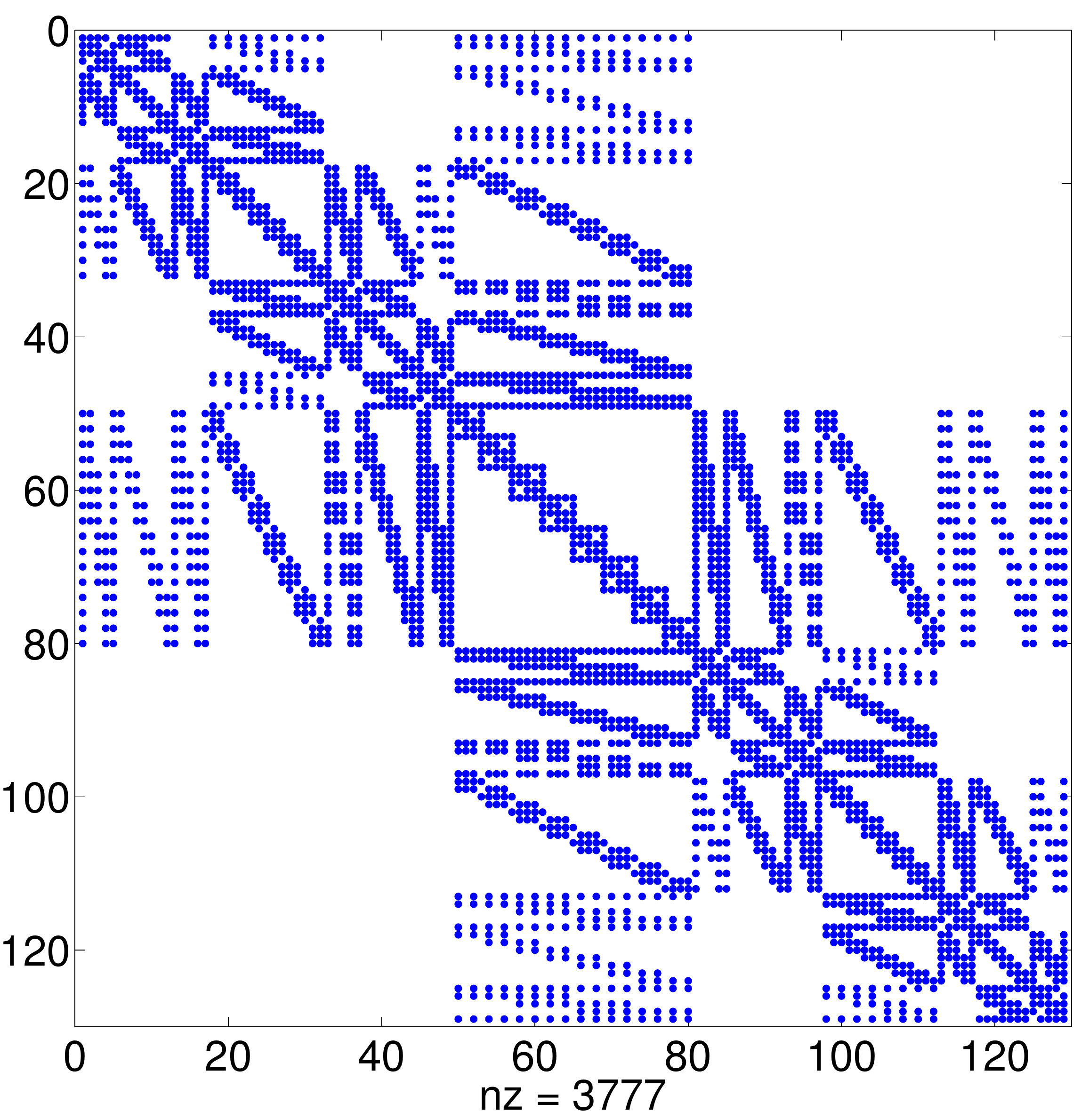}
 \caption{Sparsity patterns  for the multiscale
   method using the basis in~\eqref{eq:comb basis} with $N=32$.
   Left: unknowns  ordered as~\eqref{eq:comb basis}. Right:
   lexicographical ordering.} 
\label{fig:spymulti}
\end{figure}
\begin{figure}[!h]
 \centering
 \includegraphics[width=4.5cm]{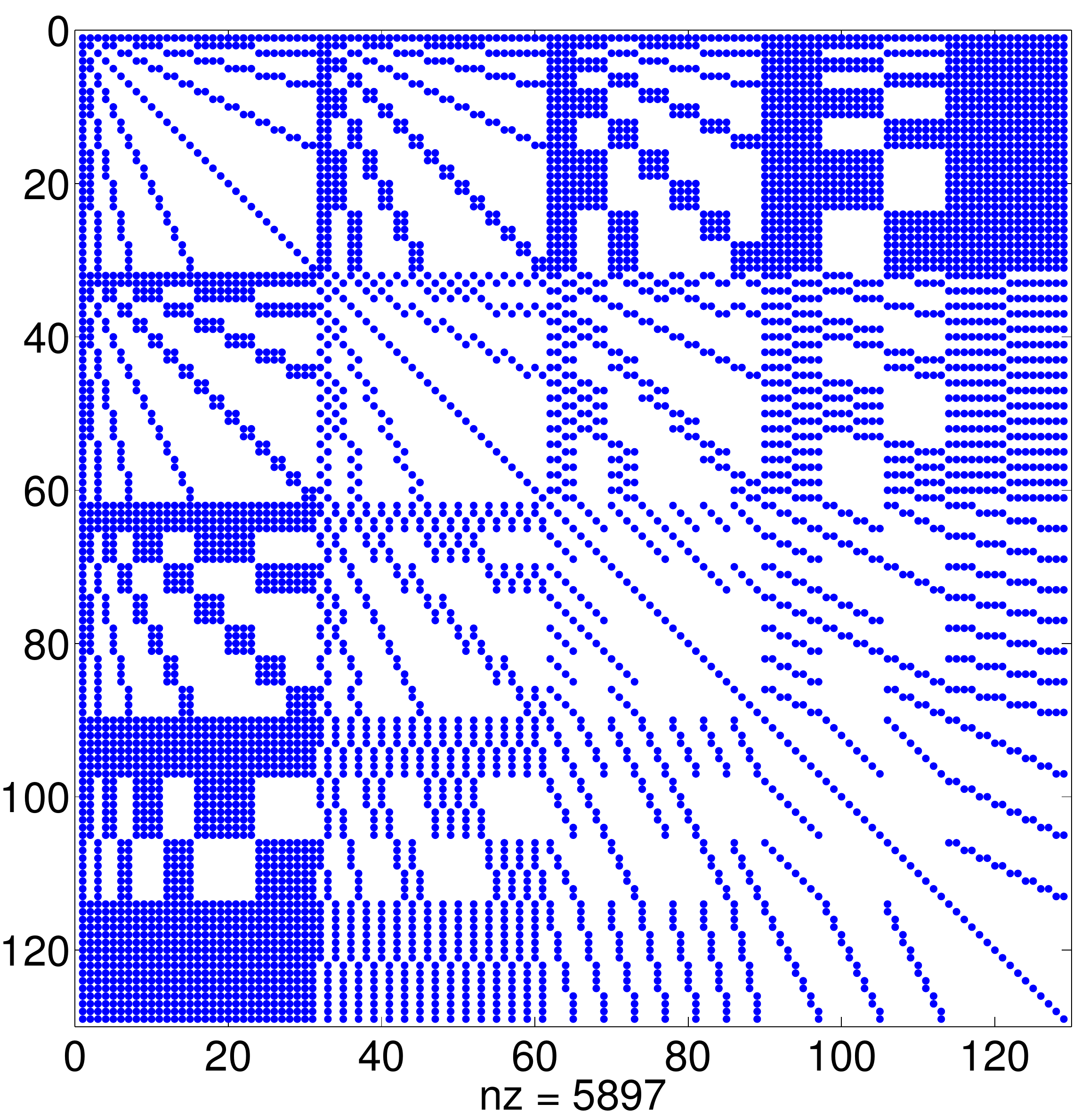}\quad
 \includegraphics[width=4.5cm]{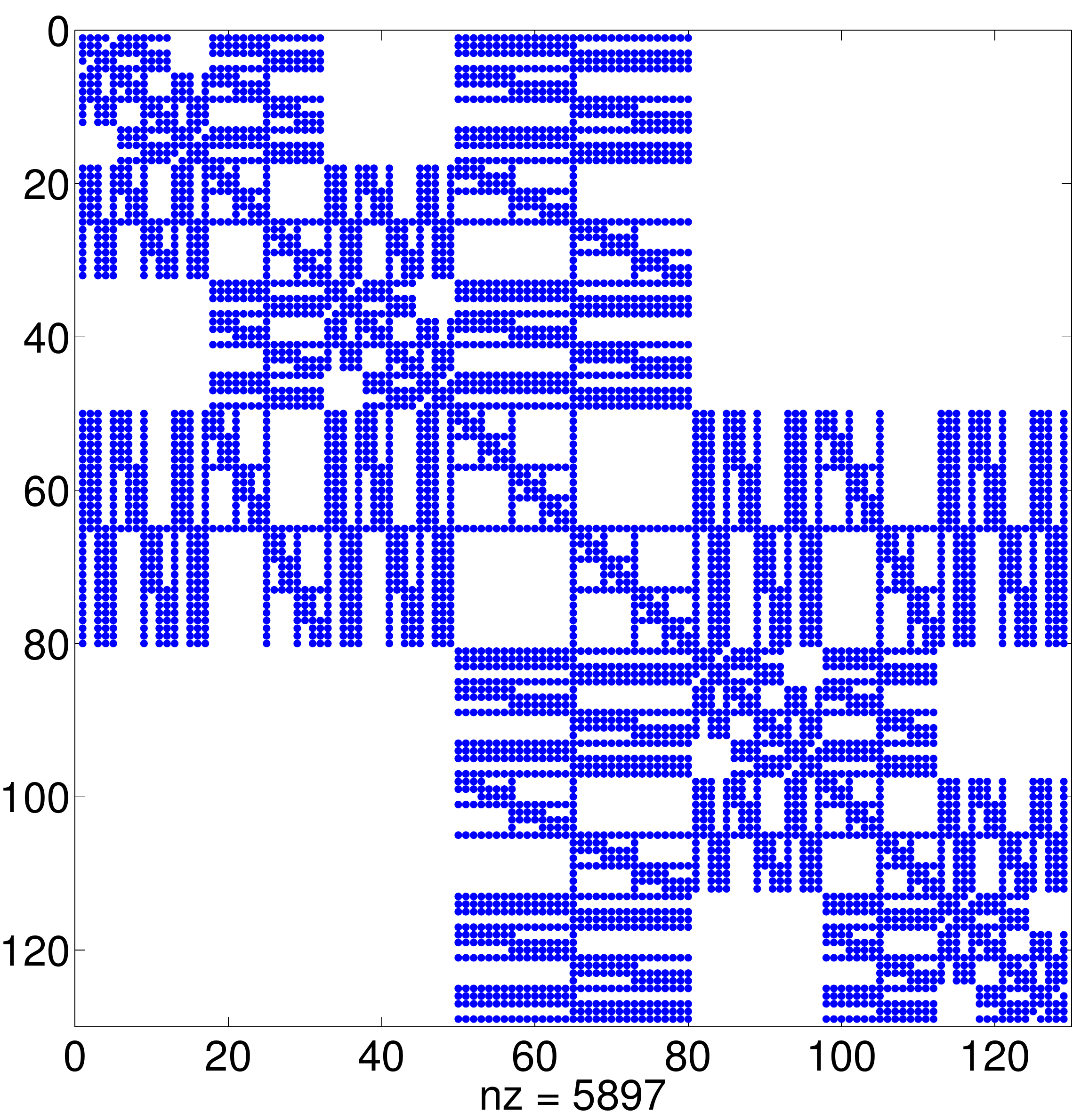}
 \caption{Sparsity patterns  using a hierarchical basis~\eqref{eq:comb hier}
   with $N=32$. Left: unknowns ordered as in~\eqref{eq:comb
     hier}. Right: lexicographical ordering.}
\label{fig:spymultihier}
\end{figure}

\section{Comparison of three FEMs}
\label{sec:good times}
In \autoref{fig:MultiScale ROC}
we presented results that
 demonstrated that the three methods achieve similar levels of
 accuracy. We now wish to quantify if, indeed, the sparse grid
 methods are more efficient than the classical method. To do this in a
 thorough fashion would require a great deal of effort to
\begin{itemize}
\item investigate efficient storage of matrices arising from these
  specialised grids;
\item investigate the design of suitable preconditioners for iterative
  techniques, or  Multigrid methods.
\end{itemize}
These topics are beyond the scope of this %introductory
article. 
Instead, we apply a direct solver, since it is the simplest possible ``black-box'' solver 
and avoids concerns involving 
preconditioners and stopping criteria. Therefore we compare the methods
with respect to the wall-clock time taken by MATLAB's ``backslash''
solver. Observing the diagnostic information provided
(\texttt{spparms('spumoni',1)}), we have verified that this defaults
to the CHOLMOD solver for symmetric positive definite matrices
\cite{Davi08}.

The results  we present  were generated on a single node of a 
Beowulf cluster, equipped with 32 Mb RAM and two  AMD Opteron 2427,
2200 MHz processors, each with 6 cores. The efficiency of parallelised
linear algebra routines can be highly dependent on the matrix
structure. Therefore we present results obtained both using all 12
cores, and using a single core (enforced by 
launching MATLAB with the \texttt{-singleCompThread} option).
All times reported are in seconds, and  have been averaged over three runs.

In \autoref{tab:compare} we can see that, for $N=2^{12}$, over a thousand
seconds are required to solve the system for the classical method on a
single core, and 560 seconds with all 12 cores enabled. (It is
notable, that, for smaller problems,  the solver was more efficient when
using just 1 core).

When the two-scale method is used,  \autoref{tab:compare} shows that
there is no great loss of accuracy, but solve times are reduced by a
factor of 3 on 1 core and (roughly) a factor of 5 on 12
cores. Employing the multiscale method, we see a speed-up of (roughly) 7
on one core, and 15 on 12 cores. %\texttt{Lots of diagrams to follow}.

\begin{table}[htb]
\caption{Comparing  efficiency of the classical, two-scale,
  and multiscale FEMs}
\label{tab:compare}
\centering
\begin{tabular}{l|r|r|r|r}
 \multicolumn{1}{c}{}&
 \multicolumn{4}{c}{Classical Galerkin}\\ \hline   
$N$ &       512 &      1024 &      2048 &      4096 \\ \hline
                Error & 4.293e-03 & 2.147e-03 & 1.073e-03 & 5.346e-04 \\
  Solver Time (1 core) &     3.019 &    18.309 &   134.175 &  1161.049 \\
Solver Time (12 cores) &     6.077 &    26.273 &   119.168 &   562.664 \\
    Degrees of Freedom &    261,121 &   1,046,529 &   4,190,209 &  16,769,025 \\
   Number of Non-zeros &   2,343,961 &   9,406,489 &  37,687,321 & 150,872,089 \\
 \hline \hline
 \multicolumn{1}{c|}{}&
 \multicolumn{4}{c}{Two-scale method}\\ \hline
            $N$ &       512 &      1000 &      2197 &      4096 \\ \hline
                 Error & 4.668e-03 & 2.383e-03 & 1.082e-03 & 5.796e-04 \\
  Solver Time (1 core) &     0.806 &     4.852 &    43.702 &   281.838 \\
Solver Time (12 cores) &     0.810 &     3.274 &    21.448 &   105.727 \\
    Degrees of Freedom &     7,105 &    17,901 &    52,560 &   122,625 \\
   Number of Non-zeros & 1,646,877 & 6,588,773 &33,234,032 & 118,756,013 \\
 \hline \hline
 \multicolumn{1}{c|}{}&
 \multicolumn{4}{c}{Multiscale method}\\ \hline
                   $N$ &       512 &      1024 &      2048 &      4096 \\ \hline
                 Error & 4.437e-03 & 2.219e-03 & 1.109e-03 & 5.529e-04 \\
  Solver Time (1 core) &     0.370 &     2.280 &    20.692 &   150.955 \\
Solver Time (12 cores) &     0.384 &     1.665 &     7.464 &    36.469 \\
    Degrees of Freedom &     4,097 &     9,217 &    20,481 &    45,057 \\
   Number of Non-zeros &   996,545 & 3,944,897 &15,525,569 &61,585,473 \\
 \end{tabular}
\end{table}

%\clearpage
\begin{figure}[htb]
\centering
\includegraphics[height=4cm]{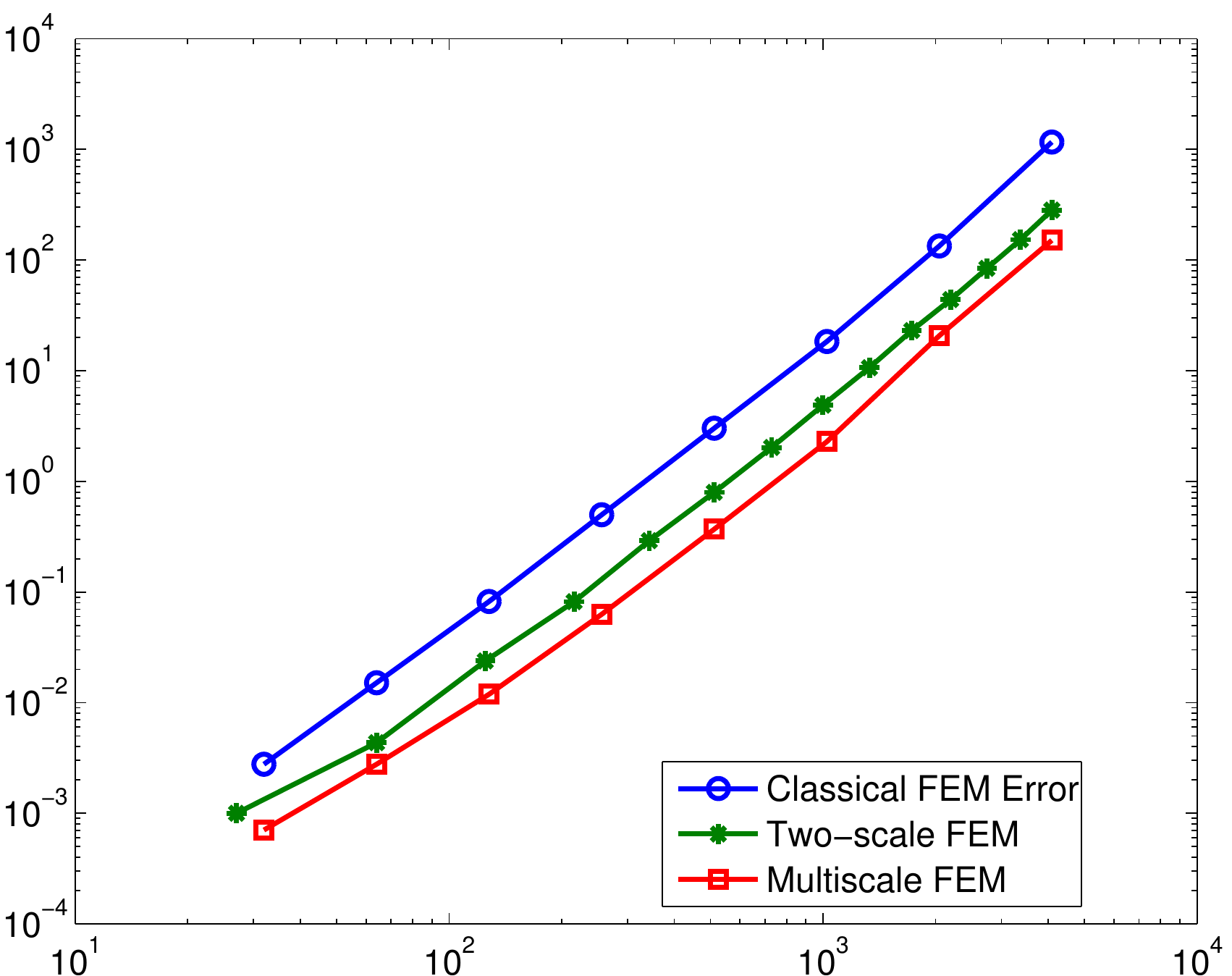} ~
\includegraphics[height=4cm]{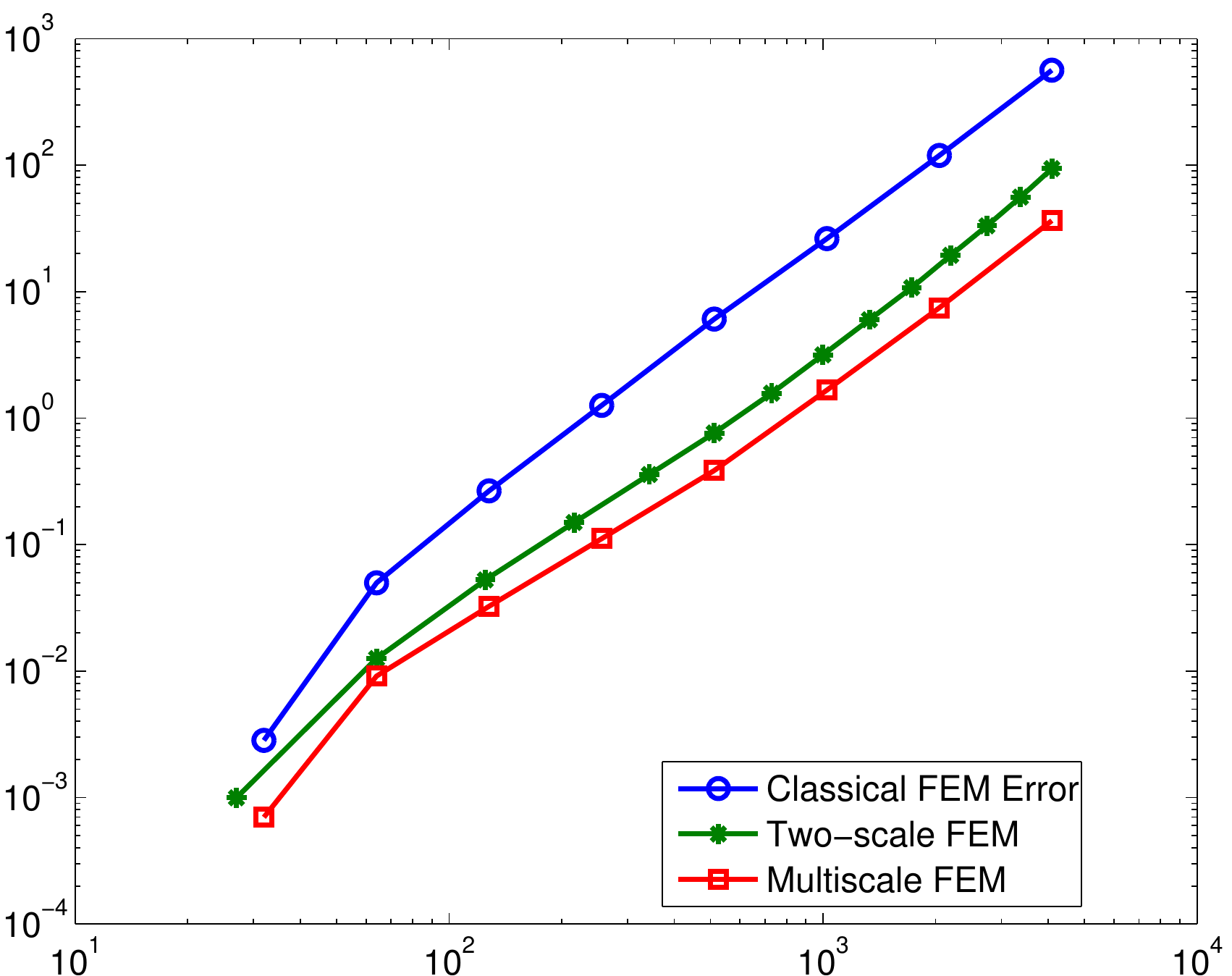}
\caption{Solve time, in seconds, for  linear systems using a direct
  solver on 1 core (left) and 12 cores (right)}
\label{fig:solve times}
\end{figure}
\begin{figure}[h!]
\centering
\includegraphics[height=4cm]{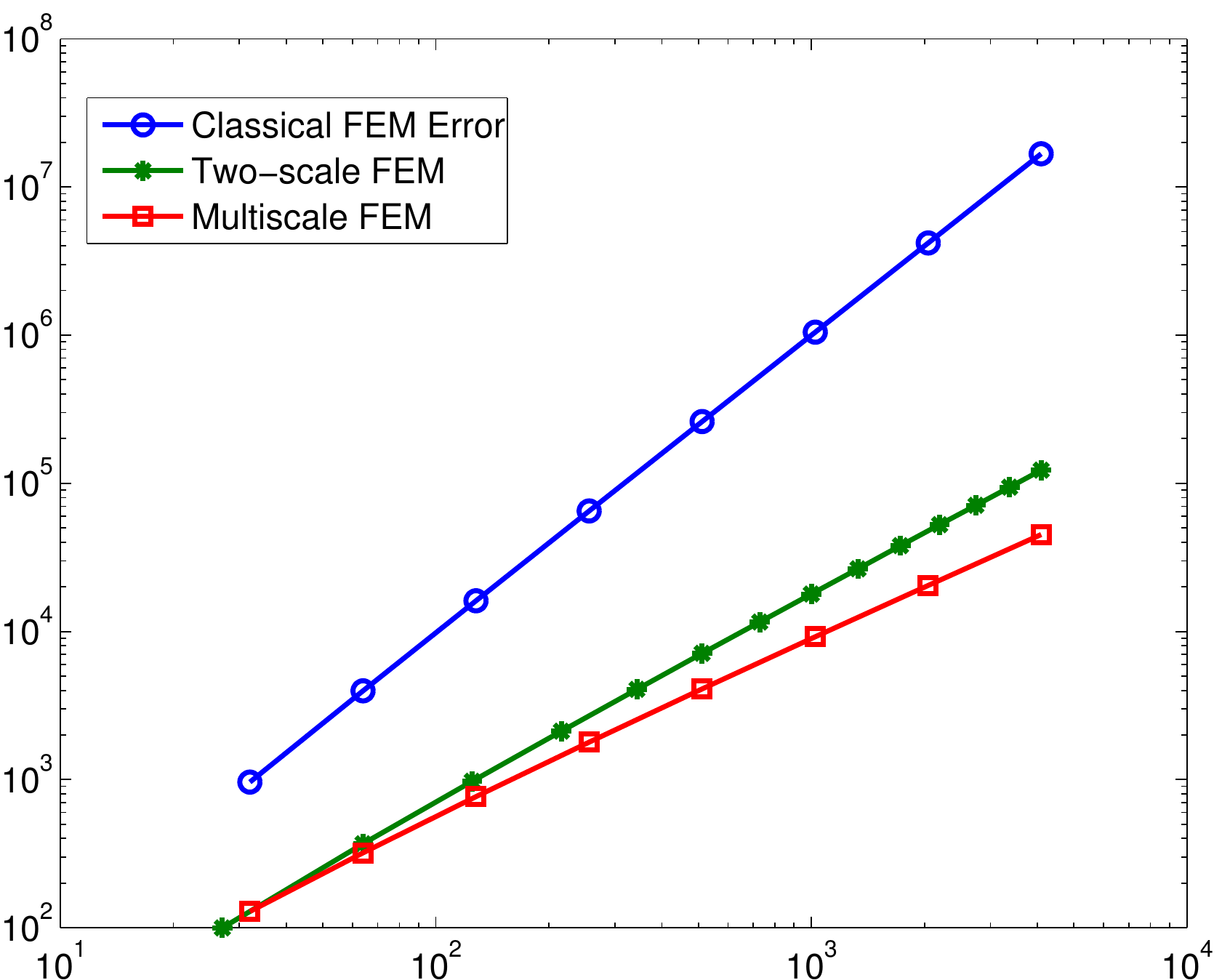} ~
\includegraphics[height=4cm]{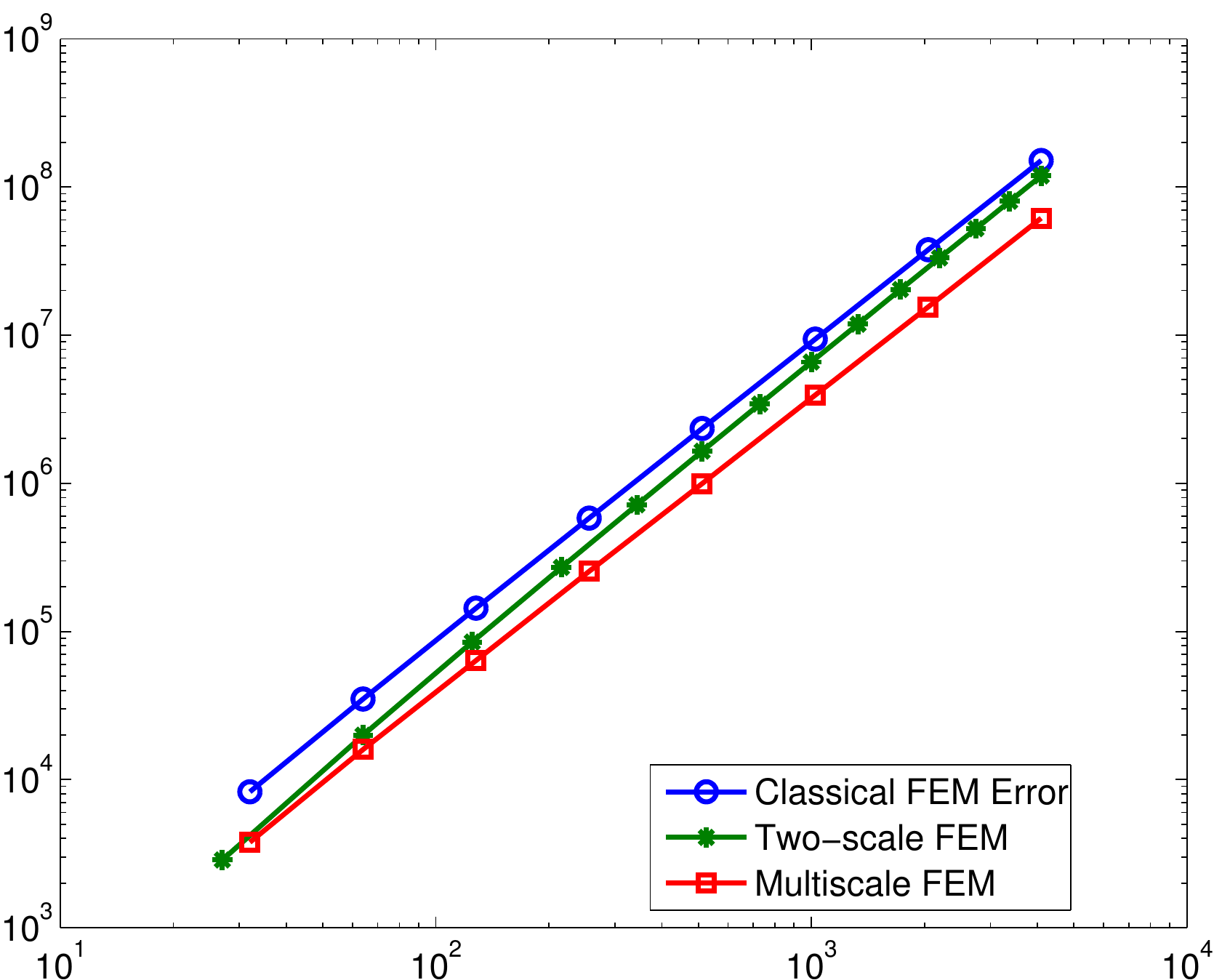}
\caption{The number of degrees of freedom (left) 
and non-zero entries in
 the system matrices (right)}
\label{fig:DOFs and NNZs}
\end{figure}

\section{Conclusions}
\label{sec:conclusions}
There are many published papers on the topic of
sparse grid methods, 
particularly over the last two decades. Often though, they deal with
highly specialized problems, 
and the analysis tends to be directed at readers that are highly
knowledgeable in the fields of  
PDEs, FEMs, and functional
analysis. Here we have presented a style of analysing sparse grid
methods in a way that is  accessible to a  
more general audience. We have also provided snippets of code and
details on how to   
implement these methods in MATLAB.

The methods presented  are not new. However, by treating the usual
multiscale sparse grid method as a generalization of the two-scale
method, we have presented a conceptually simple, yet rigorous, 
 way of understanding and analysing it.

Here we have treated just a  a simple two-dimensional  problem on a
uniform mesh. However, 
this approach  has been employed to analyse
sparse grid methods for specialised problems whose solutions feature
boundary layers, and that are solved on layer-adapted meshes, 
see \cite{RuMa14,RuMa15}.

More interesting generalisations are possible, but the most important
of these is to higher dimensional problems. An exposition of the
issues involved is planned.

\subsection*{Acknowledgements}
The work of Stephen Russell is supported by a fellowship from the
College of Science at the National University of
Ireland, Galway. The simple approach in \eqref{eq:error trick}
in determining the error was shown to us by Christos
Xenophontos.% of the University of Cyprus.

\bibliographystyle{plain}
\bibliography{stephen}
\end{document}